\documentclass[12pt, a4paper]{article}
\usepackage{hyperref}
\usepackage{amsmath,amsfonts,amsthm,amssymb,dsfont,natbib}
\usepackage[utf8]{inputenc}
\usepackage[T1]{fontenc}

\theoremstyle{plain}
\newtheorem{theorem}{Theorem}[section]
\newtheorem{lemma}[theorem]{Lemma}
\newtheorem{corollary}[theorem]{Corollary}
\newtheorem{proposition}[theorem]{Proposition}

\theoremstyle{definition}
\newtheorem{example}[theorem]{Exemple}
\newtheorem{remark}[theorem]{Remark}

\theoremstyle{remark}

\newcommand{\rmd}{\mathrm{d}}

\begin{document}

\title{The coupling method in extreme value theory}
\author{Benjamin Bobbia$^*$, Clément Dombry$^*$, Davit Varron\footnote{Universit{\'e} Bourgogne Franche-Comt{\'e}, Laboratoire de Math{\'e}matiques de Besan\c{c}on, UMR CNRS 6623, 16 route de Gray,  25030 Besan{\c c}on Cedex, France.}}
\maketitle
\date{}

\begin{abstract}
A coupling method is developed for univariate extreme value theory, providing an alternative to the use of the tail empirical/quantile processes. Emphasizing the Peak-over-Threshold approach that approximates the distribution above high threshold by the Generalized Pareto distribution, we compare the empirical distribution of exceedances and the empirical distribution associated to the limit Generalized Pareto model and provide sharp bounds for their Wasserstein distance in the second order Wasserstein space. As an application, we recover standard results on the asymptotic behavior of the Hill estimator, the Weissman extreme quantile estimator or the probability weighted moment estimators, shedding some new light on the theory.
\end{abstract}

\section{Introduction}
The purpose of extreme value theory (EVT) is to make statistical inference on the tail region of a probability distribution, e.g. assess rare event probabilities or high order quantiles. In the extreme regime of interest, only a limited number of observations are usually available and classical non parametric methods break down due to the lack of relevant data. To circumvent this issue, extreme value theory assumes some kind of regularity in the tail of the distribution that makes extrapolation possible and allows to exploit moderately high observations for inference in the extreme regime. In the univariate setting, this results in semi-parametric methods where the distribution tail is approximated by a parametric model such as the generalized extreme value (GEV) or generalized Pareto (GP) distribution. 
 Therefore, EVT is essentially an asymptotic theory where the statistical analysis focuses on $k$ extreme observations in a sample of size $n$, for an intermediate sequence $k=k(n)$ satistfying $k\to\infty$ and $k/n\to 0$ as $n\to\infty$.
 
 In the Peak-over-Threshold (PoT) method, the $k$ largest order statistics are considered; suitably rescaled, they follow approximately the GP distribution and the probability weighted moment (PWM) or maximum likelihood (ML) are classically used to estimate the tail parameters and extreme quantiles
 (see e.g. \cite{D98}, \cite{dHF06} Section 2.3).
In the block maxima (BM) method, the data sample is divided into $k$ blocks of size $m=\lfloor n/k\rfloor$ and the suitably rescaled block maxima follow approximately a GEV distribution whose parameters can again be estimated with PWM or ML method \citep{FdH15,BS17,BS18,DF19}. In both cases, one has to deal with the fundamental misspecification inherent with EVT: the extreme sample follows only approximately the limiting model (GEV or GP). This framework also raises the difficult practical issue of the choice of the effective sample size $k$: larger values of $k$ provide a smaller estimation variance, but also a poorer approximation by the limit model leading to a higher estimation bias.
There is a rich literature on methods for threshold selection and bias reduction but we do not address these issues in the present paper. For a general background on extreme value statistics, the reader should refer to the monographs by \cite{BGTS04} or \cite{dHF06}. 

Many results in EVT rely on the theory of empirical processes and use the tail empirical process or the tail quantile process. The purpose of the present paper is to develop an alternative approach emphasizing coupling arguments between the effective extreme sample and an ideal sample from the limiting model. Using the formalism of Wasserstein spaces, we obtain quantitative results measuring the quality of the approximation of the pre-limit model by the limiting model. As a first illustration, we give (in the unbiased case) a simple and elegant proof of the asymptotic normality of the celebrated Hill estimator \citep{H75} together with a non asymptotic upper bound for the Wasserstein distance to Gaussianity, see Corollary \ref{cor:Hill_AN}.

The philosophy of the coupling method in EVT is that, in the limit GP or GEV model, the behavior of the statistic of interest is often much simpler to analyze. The corresponding result in the PoT or BM framework taking into account the misspecification inherent to EVT is often much more difficult to prove and usually requires involved empirical process theory with weighted norms. We present here how the coupling method can be used to carry over the result from the limit model to the pre-limit model using Wasserstein distance estimates, which works quite smoothly under Lipschitz assumptions. 
To be more specific, the asymptotic normality of the Hill estimator for Pareto random variables is a straightforward application of the central limit theorem. The PWM based estimators in the GEV and GP models were considered by \cite{HWW85} and \cite{HW87} respectively and their analysis relies on the general theory of $L$-estimation for i.i.d. samples. Using the coupling method, these results can be carried over to the PoT or BM framework quite easily, offering an original and conceptually important point of view. 

Let us mention that the last decades have seen a fast development of the use of coupling methods, Wasserstein spaces and optimal transport in statistics. The reader should refer to the excellent review by \cite{PZ19} and the reference therein. In this paper, we see for the first time how these insightful and powerful notions apply successfully in EVT.

The structure of the present paper is as follows. In Section~\ref{sec:2}, we introduce the required background on coupling theory and Wasserstein spaces. Then, in a general framework not related to EVT, we state in Theorem~\ref{theo:wasserstein-emp-measure} the equality of the Wasserstein distance (in the second order Wasserstein space) between the distributions of the empirical measures of two i.i.d. samples with the same size and the Wasserstein distance (in the first order Wasserstein space) of the underlying distributions generating the samples. The proof uses elementary coupling arguments together with the more subtle Kantorovitch duality \citep[Chapter 5]{V09}. We believe this result is quite powerful and may find application in many different situations. In Section~\ref{sec:3}, we see how it applies to EVT, mostly in the PoT framework, and introduce some background on EVT. In Section~3.2, for the sake of clarity, we first consider the heavy-tailed case and provide a sharp estimate between the empirical measure of observations above high-threshold and the empirical measure associated with i.i.d. Pareto observations (Theorem~\ref{thm:Wass_dist_Pi_EVT}). Statistical consequences for the Hill estimator and Weissman extreme quantile estimates are discussed in Corollaries~\ref{cor:Hill_AN} and~\ref{cor:Weissman_AN}. Further extensions are considered in Section~3.3: the asymptotic regime where bias occurs, the generalization to all domain of attractions and some results in the BM framework. All the proofs are gathered in Section~4.

We finally introduce some notations used throughout the paper. Given two real-valued sequences $(u_n)_{n\geq 1}$ and $(v_n)_{\geq 1}$, the notation $u_n=o(v_n)$ (respectively $u_n=O(v_n)$) means that $u_n=\varepsilon_n v_n$ for some sequence $(\varepsilon_n)_{n\geq 1}$ converging to $0$ (resp. bounded). Similarly, given two sequences of random variables $(U_n)_{n\geq 1}$ and $(V_n)_{n\geq 1}$, the notation $U_n=o_P(V_n)$ (resp. $U_n=O_P(V_n)$) means that $U_n=\varepsilon_n V_n$ for some sequence of random variables $(\varepsilon_n)_{n\geq 1}$ converging to $0$ in probability (resp. bounded in probability). The notations $\stackrel{d}=$ and $\stackrel{d}\longrightarrow$ stand respectively for equality in distribution and convergence in distribution. All random variable are defined on some underlying probability space $(\Omega,\mathcal{F},\mathbb{P})$ and for $p\in [1,\infty]$,  we denote by $\|\cdot\|_{L^p}$ the usual norm on $L^p(\Omega,\mathcal{F},\mathbb{P})$.

\section{Coupling and sampling}\label{sec:2}
We introduce some background on coupling theory and Wasserstein spaces as well as general results about coupling and sampling. More precisely, we evaluate the Wasserstein distance between the empirical distributions of two independent and identically distributed (i.i.d.) samples and between statistics thereof. 

In the following, we consider a metric space $(\mathcal{X},d)$ endowed with its Borel $\sigma$-algebra $\mathcal{B}(\mathcal{X})$. We denote by $\mathcal{M}(\mathcal{X})$ the set of probability measures on $(\mathcal{X},\mathcal{B}(\mathcal{X}))$ and by $\varepsilon_x$ the Dirac mass at $x\in\mathcal{X}$. 

\subsection{Background on coupling and Wasserstein spaces}
The coupling method has long been an important tool in probability theory, see e.g. the monographs by \cite{L92} and \cite{T00}. A coupling between two probability measures $P_1,P_2\in \mathcal{M}(\mathcal{X})$ is a pair of $\mathcal{X}$-valued random variables $X_1,X_2$ defined on a common probability space and such that $X_1\rightsquigarrow P_1$ and $X_2\rightsquigarrow P_2$.
This notion of coupling is crucial in the definition of the Wasserstein distance which is a powerful tool in statistics; it is for instance central in the analysis of the bootstrap by \cite{BF81}. We provide below some basic properties of Wasserstein spaces, more details are to be found in \cite{V09}. 

The Wasserstein distance of order $p\in [1,\infty]$ between two probability measures $P_1,P_2\in \mathcal{M}_1(\mathcal{X})$ is defined by
\[
W_p(P_1,P_2)=\inf\Big{\{ }\| d(X_1,X_2)\|_{L^p}\,:\, X_i\rightsquigarrow P_i,\ i=1,2\Big{\}}.
\]
When $(\mathcal{X},d)$ is complete and separable, the infimum in this definition is achieved and there exist optimal couplings, i.e. couplings $(X_1,X_2)$ such that $W_p(P_1,P_2)=\| d(X_1,X_2)\|_{L^p}$. The Wasserstein space of order $p$ is defined as the set of probability measures 
\[
\mathcal{W}_p(\mathcal{X})=\left\{P\in\mathcal{M}(\mathcal{X}) \,:\, \int_{\mathcal{X}}d(x_0,x)^pP(\rmd x)<\infty \right\},
\]
where $x_0\in\mathcal{X}$ denotes some origin, i.e. some fixed point whose choice is irrelevant. Note that $(\mathcal{W}_p(\mathcal{X}),W_p)$ is itself a metric space, which is complete and separable if both $p\in[1,\infty)$ and $(\mathcal{X},d)$ is itself complete and separable. Also note that $W_p(P_n,P)\to 0$ is equivalent to the convergence in distribution $P_n\stackrel{d}{\to} P$ together with the convergence of moments \[
\int_{\mathcal{X}}d(x_0,x)^p P_n(\rmd x)\to \int_{\mathcal{X}}d(x_0,x)^pP(\rmd x)\quad \mbox{ for some (all) $x_0\in\mathcal{X}$}.
\]

\begin{example}\label{ex:1} When $\mathcal{X}=\mathbb{R}^d$ is endowed with the Euclidean norm, the Wasserstein space of order $p<\infty$ simply consists in all probability measures with finite moment of order $p$, i.e.
\[
\mathcal{W}_p(\mathbb{R}^d)=\left\{P\in\mathcal{M}(\mathbb{R}^d) \,:\, \int_{\mathbb{R}^d}\|x\|^pP(\rmd x)<\infty \right\}.
\]
For $p=\infty$, $\mathcal{W}_\infty(\mathbb{R}^d)$ is the set of probability measures with bounded support. In general, the Wasserstein distance between probability measures on $\mathbb{R}^d$ cannot be computed explicitly, except in the unidimensional case. When $d=1$, optimal couplings are provided by the probability integral transform. For $P\in\mathcal{M}(\mathbb{R})$, the random variable $X=F^\leftarrow(U)$ has distribution $P$, where $F^\leftarrow$ denotes the quantile function of $P$ and $U$ a random variable with uniform distribution on $[0,1]$. Using obvious notations, $X_1=F_1^\leftarrow(U)$ and $X_2=F_2^\leftarrow(U)$ is hence a coupling between $P_1,P_2\in\mathcal{M}(\mathbb{R})$ and it turns out that this coupling is optimal so that
\[
W_p(P_1,P_2)=
\left(\int_0^1|F_1^\leftarrow(u)-F_2^\leftarrow(u)|^p \,\mathrm{d}u\right)^{1/p},\quad \mbox{if $p\in [1,\infty)$},
\]
and $W_\infty(P_1,P_2)=\sup_{u\in (0,1)} |F_1^\leftarrow(u)-F_2^\leftarrow(u)|$.
\end{example}

\begin{example}\label{ex:2nd-order-Wasserstein-space}
Given a metric space $(\mathcal{X},d)$, the Wasserstein space $(\mathcal{W}_p(\mathcal{X}),W_p)$ is also a metric space so that we can iterate the construction and consider the \textit{second order Wasserstein space} $\mathcal{W}_p(\mathcal{W}_p(\mathcal{X}))$. An element $\mathcal{P}\in\mathcal{W}_p(\mathcal{W}_p(\mathcal{X}))$ can be seen as the distribution of a random measure $\Pi$ on $(\mathcal{X},d)$ satisfying the integrability condition $\mathbb{E}\left[ \int_{\mathcal{X}}d(x_0,x)^p\Pi(\rmd x)\right] <\infty$. In the next section, we will consider the empirical distribution of random samples so that random measures and second order Wasserstein spaces will naturally arise. The Wasserstein distance on $\mathcal{W}_p(\mathcal{W}_p(\mathcal{X}))$ is denoted by $W_p^{(2)}$ and defined by
\[
W_p^{(2)}(\mathcal{P}_1,\mathcal{P}_2)=\inf\Big{\{}\|W_p(\Pi_1,\Pi_2)\|_{L^p}
\,:\, \Pi_i\rightsquigarrow \mathcal{P}_i,\ i=1,2 \Big{\}}.
\]
Note that \cite{LGL19} have considered second order Wasserstein spaces when discussing Wasserstein barycenters for random measures.
\end{example}

\subsection{Wasserstein distance between empirical distributions}
Consider two i.i.d. samples $X_1,\ldots,X_n$ and $X_1^*,\ldots,X_n^*$, with size $n\geq 1$, taking values in $(\mathcal{X},d)$ and with distribution $P$ and $P^*$ respectively. The corresponding empirical distributions are defined by
\[
\Pi_n=\frac{1}{n}\sum_{i=1}^n \varepsilon_{X_i}
\quad\mbox{and}\quad
\Pi_n^*=\frac{1}{n}\sum_{i=1}^n \varepsilon_{X_i^*}.
\]
These are random measures on $(\mathcal{X},d)$ satisfying 
\[
\mathbb{E}\left[ \int_{\mathcal{X}}d(x_0,x)^p\Pi_n(\rmd x)\right]=\int_{\mathcal{X}}d(x_0,x)^pP(\rmd x).
\]
and similarly with $\Pi_n^*$ and $P^*$. The expectation is finite if and only if $P\in\mathcal{W}_p(\mathcal{X})$ and, according to Example~\ref{ex:2nd-order-Wasserstein-space}, the distribution of $\Pi_n$ is then an element of the second order Wasserstein space $\mathcal{W}_p(\mathcal{W}_p(\mathcal{X}))$. The following Theorem states that the Wasserstein distance between the distributions of the empirical measures $\Pi_n$ and $\Pi_n^*$ is equal to the Wasserstein distance between the distributions $P$ and $P^*$ that have generated the samples.

\begin{theorem}\label{theo:wasserstein-emp-measure}
Assume $(\mathcal{X},d)$ is complete and separable. Let $p\in[1,\infty)$ and $P,P^*\in\mathcal{W}_p(\mathcal{X})$. Then,
\begin{equation}\label{eq:wasserstein-emp-measure-1}
W_p^{(2)}\Big{(}P_{\Pi_n},P_{\Pi_n^*}\Big{)}= W_p(P,P^*).
\end{equation}
\end{theorem}

\begin{remark}
As will be stated in the proof, the inequality
\[
W_p^{(2)}\Big{(}P_{\Pi_n},P_{\Pi_n^*}\Big{)}\leq W_p(P,P^*)
\]
always holds true, even with $p=\infty$, and relies on elementary coupling arguments. The assumptions $p<\infty$ and $(\mathcal{X},d)$ complete and separable are required for the converse inequality whose proof uses the Kantorovitch duality from optimal transport, see \citet[Chapter 5]{V09}.
\end{remark}

\begin{remark} Theorem~\ref{theo:wasserstein-emp-measure} can be generalized to samples with random size or weighted samples with random weights. Consider $(X_i)_{i\geq 1}$ and $(X_i^*)_{i\geq 1}$ sequences of i.i.d. random variables with distribution $P$ and $P^*$ respectively. Let $(w_i)_{i\geq 1}$ be a random sequence of non-negative weights summing up to $1$ and independent of the sequences of observations $(X_i)_{i\geq 1}$ and $(X_i^*)_{i\geq 1}$. 
The weighted empirical distributions are defined by
\[
\Pi=\sum_{i\geq 1} w_i \varepsilon_{X_i}
\quad\mbox{and}\quad
\Pi^*=\sum_{i\geq 1} w_i \varepsilon_{X_i^*}
\]
and are random elements in $\mathcal{W}_p(\mathcal{X})$.
Then, Theorem~\ref{theo:wasserstein-emp-measure} still holds and the proof is readily adapted. The case of a sample with random size $N$ corresponds to $w_i=1/N$ for $i\leq N$ and $w_i=0$ otherwise. Multinomial weights naturally appear when considering bootstrap procedures. 
\end{remark}

\subsection{Wasserstein distance between sample statistics}
In many cases, one is interested in statistics $S_n=S(X_1,\ldots,X_n)$ and $S_n^*=S(X_1^*,\ldots,X_n^*)$. Due to the exchangeability of i.i.d. samples, it is natural to consider symmetric statistics that can be written as a functional of the empirical distributions, that is $S_n=S(\Pi_n)$ and $S_n=S(\Pi_n^*)$ -- the same letter $S$ is used without risk of confusion. The following corollary of Theorem~\ref{theo:wasserstein-emp-measure} provides a simple upper bound under a natural Lipschitz condition.
\begin{corollary}\label{cor:lipschitz-statistic}
Let $p\in [1,\infty]$ and $P,P^*\in\mathcal{W}_p(\mathcal{X})$. Let $(\mathcal{Y},\delta)$ be a metric space and $S:\mathcal{W}_p(\mathcal{X})\to \mathcal{Y}$ be a Lipschtiz functionnal with Lipschitz constant $\mathrm{Lip}(S)$.
Then, the statistics $S_n=S(\Pi_n)$ and $S_n^*=S(\Pi_n^*)$ satisfy
\[
W_p\Big{(}P_{S_n},P_{S_n^*}\Big{)}\leq \mathrm{Lip}(S) W_p(P,P^*),
\]
where the Wasserstein distance in the left-hand side is taken on the space $\mathcal{W}_p(\mathcal{Y})$.
\end{corollary}

In extreme value theory, important examples of Lipschitz statistics are the moments and probability weighted moments, as stated in the next proposition.
\begin{proposition}\label{prop:lipschitz-statistic}
\begin{enumerate}
\item[i)] Let $p\in [1,\infty]$ and $q\in [1,p]$ finite. Consider the moment functional defined, for $\pi\in \mathcal{W}_p(\mathcal{X})$, by
\[
S(\pi)=\Big(\int_{\mathcal{X}}d(x_0,x)^q \pi(\rmd x)\Big)^{1/q}.
\]
Then, $S:\mathcal{W}_p(\mathcal{X})\to\mathbb{R}$ is Lipschitz with $\mathrm{Lip}(S)=1$. 
\item[ii)] Let $p\in [1,\infty]$, $q\in [1,p]$ finite and $r,s\geq 0$. Consider the probability weighted moment functional defined, for $\pi\in \mathcal{W}_p(\mathbb{R})$, by
\[
S(\pi)=\Big(\int_{0}^1 F_\pi^\leftarrow(u)^q u^r(1-u)^s \rmd u\Big)^{1/q}
\]
with $F_\pi^\leftarrow$ the quantile function of $\pi$. Then, $S:\mathcal{W}_p(\mathbb{R})\to\mathbb{R}$ is Lipschitz with $\mathrm{Lip}(S)\leq 1$.
\end{enumerate}
\end{proposition}

\section{Wasserstein distance estimates in extreme value theory}\label{sec:3}
\subsection{Background on univariate extreme value theory}
Two main approaches exist in univariate extreme value theory: the Block Maxima (BM) method and the Peaks-over-Threshold (PoT) method. The Fisher-Tipett-Gnedenko theorem \citep{FT28,G43} is the basis of the BM method and states that, if a distribution $F$ satisfies
\begin{equation}\label{eq:cv-GEV}
F^n(a_n\cdot+b_n)\stackrel{d}\rightarrow G(\cdot)
\end{equation}
for some non-degenerate limit distribution $G$ and normalizing sequences $a_n>0$ and $b_n\in\mathbb{R}$, then, up to location and scale, the limit distribution $G$ is necessarily equal to an extreme value distribution $G_\gamma$ defined by
\begin{equation}\label{eq:def-GEV}
G_\gamma(x)=\exp(-(1+\gamma x)^{-1/\gamma}),\quad 1+\gamma x>0. 
\end{equation}
The parameter $\gamma\in\mathbb{R}$ is called the extreme value index. 
The case $\gamma=0$ corresponds to the limit as $\gamma\to 0$ and $G_0$ is the Gumbel distribution $G_0(x)=\exp(-\exp(-x))$. The link with block maxima is that Equation~\eqref{eq:cv-GEV} is equivalent to the convergence of the rescaled maximum 
\[
a_n^{-1}\Big(\max_{1\leq i\leq n} X_i -b_n\Big) \stackrel{d}\longrightarrow G \quad \mbox{as $n\to\infty$},
\]
with $X_1,\ldots,X_n$ i.i.d. random variables with common distribution $F$. When Equation~\eqref{eq:cv-GEV} holds, we say that $F$ belongs to the max-domain of attraction of $G$, noted $F\in\mathrm{D}(G)$. 

The characterization of the domain of attraction of $G_\gamma$ is due to \cite{G43} in the case $\gamma\neq 0$ and to \cite{dH71} in the case $\gamma=0$:  $F\in\mathrm{D}(G_\gamma)$ if and only if the tail quantile function $U(t)=F^\leftarrow(1-1/t)$, $t>1$, satisfies the first order condition
\begin{equation}\label{eq:first-order-condition}
\lim_{t\to\infty} \frac{U(tx)-U(t)}{a(t)} =\frac{x^\gamma-1}{\gamma},\quad x>0,
\end{equation}
for some normalizing function $a>0$.

The Peaks-over-Threshold method focuses on exceedances over a high threshold, that is observations of $X\rightsquigarrow F$ satisfying $X>u$ for $u$ closed to the right endpoint $x^*=\sup\{x:F(x)>1\}$. 
The Balkema-de Haan-Pickands Theorem \citep{BdH74,P75} states that the first order condition is also equivalent to the convergence
\begin{equation}\label{eq:cv-GPD}
\lim_{u\to x^*} \frac{1-F(u+f(u)x)}{1-F(u)} =(1+\gamma x)^{-1/\gamma},\quad 1+\gamma x>0,
\end{equation}
for some positive function $f$. A possible choice is $f(u)=a(1/(1-F(u)))$. In terms of exceedances, this writes 
\begin{equation}\label{eq:cv-GPD2}
P_{\frac{X-u}{f(u)} \mid X>u} \stackrel{d} \longrightarrow H_\gamma \quad \mbox{as $u\to x^*$},
\end{equation}
where $H_\gamma$ is the Generalized Pareto (GP) distribution 
\begin{equation}\label{eq:def-GPD}
H_\gamma(x)=1-(1+\gamma x)^{-1/\gamma},\quad 1+\gamma x>0.
\end{equation}
In the framework of regular variation theory, the study of rates of convergence usually relies on second-order regular variation, see \citet[Chapter B.3]{dHF06}. The so-called second-order condition reads
\begin{equation}\label{eq:second-order-condition2}
\lim_{t\to \infty} \frac{\frac{U(tx)-U(t)}{a(t)}-\frac{x^\gamma-1}{\gamma}}{A(t)} = \Psi_{\gamma,\rho}(x),\quad x>0,
\end{equation}
where $\rho\leq 0$ is the second order parameter, the normalizing function $A$ is regularly varying at infinity with index $\rho$, eventually negative or positive and such that $\lim_{t\to\infty}A(t)= 0$ and
\begin{equation}\label{eq:limit-second-order-condition}
 \Psi_{\gamma,\rho}(x)=\int_{1}^x s^{\gamma-1}\int_1^s u^{\rho-1}\rmd u\,\rmd s.
\end{equation}

\subsection{Analysis of the Peak-over-Threshold method (case $\gamma>0$)}
We analyze the PoT method and, for the sake of clarity, we consider first the case $\gamma>0$. One can then take $a(t)=\gamma U(t)$ so that the first order condition~\eqref{eq:first-order-condition} simplifies into standard regular variation
\begin{equation}\label{eq:RV}
\lim_{t\to\infty} \frac{U(tx)}{U(t)} =x^\gamma,\quad x>0.
\end{equation}
Equation~\eqref{eq:cv-GPD2} stating the convergence in distribution of normalized exceedances to the GP distribution is equivalent to 
\begin{equation}\label{eq:cv-Pareto}
P_{u^{-1}X \mid X>u} \stackrel{d} \rightarrow P_\alpha, \quad \mbox{as $u\to \infty$},
\end{equation}
with $P_\alpha(x)=1-x^{-\alpha}$, $x>1$, the Pareto distribution with index $\alpha=1/\gamma>0$. Furthermore, the second order condition can be simplified into
\begin{equation}\label{eq:second-order-condition}
\lim_{t\to \infty} \frac{\frac{U(tx)-U(t)}{U(t)}-x^\gamma}{A(t)} = x^\gamma \frac{x^\rho -1}{\rho},\quad x>0.
\end{equation}

\subsubsection{Pareto approximation of exceedance above high threshold}
We consider the rate of the convergence \eqref{eq:cv-Pareto} in the Wasserstein space and compare the exceedance distribution $P_{u^{-1}X \mid X>u}$ to the Pareto distribution $P_\alpha$ in $\mathcal{W}([1,\infty))$. Because the Pareto distribution has finite moments of order $p<\alpha$ only, we introduce the logarithmic distance 
\begin{equation}\label{eq:log-dist}
d(x,x')=|\log(x)-\log(x')|,\quad x,x'\in[1,\infty).
\end{equation}
We will see below that this distance is also convenient for analyzing the behavior of the Hill estimator and the Weissman quantile estimator. 

Let $t_0>1$ be such that $U(t_0)>0$. For $t\geq t_0$, we define
\begin{equation}\label{eq:prop-cv-pareto-W}
A_p(t)= \left\{\begin{array}{ll}
 \left(\int_{1}^\infty \left|\log \frac{U(zt)}{z^\gamma U(t)}\right|^p \frac{\rmd z}{z^2}\right)^{1/p}& \mbox{if $p\in [1,\infty)$},\\
 \sup_{z>1} \left|\log \frac{U(zt)}{z^\gamma U(t)}\right| & \mbox{if $p=\infty$}.
\end{array}\right.
\end{equation}
Without loss of generality, we can assume that the random variable $X\rightsquigarrow F$ is given by $X=U(Z)$ where $U(t)=F^\leftarrow (1-1/t)$ is the tail quantile function of $X$ and $Z$ follows a standard unit Pareto distribution.

\begin{proposition}\label{prop:cv-pareto-W} Let $p\in [1,\infty]$ and consider the Wasserstein space $\mathcal{W}_p([1,\infty))$ with underlying distance \eqref{eq:log-dist}.
\begin{enumerate}
\item[i)] For $t\geq t_0$,
\[
W_p(P_{U(t)^{-1}X\mid Z>t},P_\alpha)=A_p(t).
\]
\item[ii)] If $p<\infty$ and $F\in D(G_\gamma)$ with $\gamma>0$, then the function $A_p$ is bounded on $[t_0,\infty)$ and such that $\lim_{t\to\infty} A_p(t)=0$. 
\item[iii)] If the second-order condition~\eqref{eq:second-order-condition} holds, then
\[
\lim_{t\to\infty}\frac{A_p(t)}{A(t)} =c_p(\rho):= \left\{\begin{array}{cl} 
\left(\int_{1}^\infty \left|\frac{z^\rho -1}{\rho}\right|^p \frac{\rmd z}{z^2}\right)^{1/p} & \mbox{if $p\in [1,\infty)$}, \\
1/|\rho| & \mbox{if $p=\infty$ and $\rho<0$}, \\ 
+\infty & \mbox{if $p=\infty$ and $\rho=0$}. 
\end{array}\right.
\]
\end{enumerate}
\end{proposition}

\begin{remark}\label{rk:t't}
We discuss the link between $X$ and $Z$. Conditioning with respect to $X$ is more natural but conditioning w.r.t. $Z$ is mathematically much more convenient. When $F$ is continuous on $\mathbb{R}$, the two coincide since we have $\{X>U(t)\}=\{Z>t\}$.
When $F$ is not continuous at $U(t)$, the conditioning event writes 
$\{X>U(t)\}=\{Z>1/(1- F(U(t)^-))\}$ with $F(x^-)$ the left limit of $F$ at $x$. Then the conditioning event $\{Z>t\}$ is not measurable with respect to the $\sigma$-field generated by $X$ but it can be recovered introducing an extra-randomness. Let $W$ be uniform on $[0,1]$ and independent of $X$ and consider $V=F(X)+(F(X)-F(X^-))W$ and $Z=1/(1-V)$. Then $V$ has a uniform distribution on $[0,1]$ and satisfies $X=F^\leftarrow (V)$. It follows that $Z$ has a standard unit Pareto distribution and satisfies $X=U(Z)$ so that the conditioning event $\{Z>t\}$ can be written in terms of $X$ and an auxiliary random variable $W$.
\end{remark}

\subsubsection{Approximation of the empirical distribution of exceedances}
Peaks-over-Threshold inference uses the fact that the GP distribution is a good approximation for the distribution of exceedances above high threshold and all the statistics of interest are built using only observations above high threshold. It is customary to use a random threshold equal to the order statistic of order $n-k$ so that the exceedances are the $k$ top order statistics. In a first approach, a statistical procedure is often studied on the limiting model itself, that is assuming the exceedances are exactly Pareto distributed. This amounts to neglecting the misspecification inherent to extreme value theory. The error made in this approximation can be quantified by the Wasserstein distance between the empirical distributions. 

Let $X_1,\ldots,X_n$ be an i.i.d. sample with distribution $F\in D(G_\gamma)$, $\gamma>0$ and denote by $X_{1,n}\leq \ldots\leq X_{n,n}$ the order statistics.
 Define the empirical distributions of exceedances above threshold $X_{n+1-k,n}$ by
\[
\Pi_{n,k}=\frac{1}{k}\sum_{i=1}^k \varepsilon_{X_{n+1-i,n}/X_{n-k,n}}.
\]
We compare this empirical distribution to $\Pi_{k}^*=\frac{1}{k}\sum_{i=1}^k \varepsilon_{X_i^*}$, where the sample $X_1^*,\ldots,X_k^*$ is i.i.d. with Pareto distribution $P_\alpha$. 

Without loss of generality, we can assume that $X_i=U(Z_i)$ with $Z_1,\ldots,Z_n$ i.i.d. random variables with standard unit Pareto distribution with order statistics $Z_{1,n}\leq\cdots\leq Z_{n,n}$. We denote by $P_{\Pi_{n,k}\mid Z_{n-k,n}=t}$ the conditional distribution of $\Pi_{n,k}$ given $ Z_{n-k,n}=t$ and compare it to $P_{\Pi_{k}^*}$ in the second order Wasserstein space $\mathcal{W}_p(\mathcal{W}_p([1,\infty)))$ where $[1,\infty)$ is equipped with the logarithmic distance.
Following Remark~\ref{rk:t't}, note that the conditioning event $X_{n-k,n}=U(t)$ corresponds to $1/(1-F(U(t)^-))<Z_{n-k,n}\leq 1/(1-F(U(t)))$ and is equal to $T_{n-k,n}=t$ in the case $F$ is continuous at $U(t)$. Conversely, the conditioning event $Z_{n-k,n}=t$ can be expressed in terms of $X_{n-k,n}$ and an auxiliary random variable $W$.
\begin{theorem}\label{thm:Wass_dist_Pi_EVT}
Let $p\in[1,\infty)$.
\begin{enumerate}
\item[i)] Let $t_0\geq 1$ be such that $U(t_0)>0$, then
\[
W_p^{(2)}\Big{(}P_{\Pi_{n,k}\mid Z_{n-k,n}=t},P_{\Pi_{k}^*}\Big{)}=A_p(t), \quad t\geq t_0,
\]
with $A_p$ defined in Equation~\eqref{eq:prop-cv-pareto-W}.
\item[ii)] If $F\in D(G_\gamma)$ and $k=k(n)$ is an intermediate sequence, then 
\[
W_p^{(2)}\Big{(}P_{\Pi_{n,k}\mid Z_{n-k,n}},P_{\Pi_{k}^*}\Big{)} \to 0\quad \mbox{in probability}.
\]
\item[iii)] If the second order condition \eqref{eq:second-order-condition} holds, then 
\[
W_p^{(2)}\Big{(}P_{\Pi_{n,k}\mid Z_{n-k,n}},P_{\Pi_{k}^*}\Big{)}= c_p(\rho)A(n/k)(1+o_P(1)),
\]
with $c_p(\rho)$ defined in Proposition~\ref{prop:cv-pareto-W}.
\end{enumerate}
\end{theorem}
When $X$ is positive and bounded away from $0$, integration with respect to the threshold $Z_{n-k,n}$ provides the following upper bound for the Wasserstein distance between $P_{\Pi_{n,k}}$ (unconditional distribution) and $P_{\Pi_k^*}$. We denote by $\beta_{p,q}$ the density of the Beta distribution with parameter $(p,q)$.
\begin{corollary}\label{cor:Wass_dist_Pi_EVT}
If $X$ is positive and bounded away from $0$, then
\[
W_p^{(2)}\Big{(}P_{\Pi_{n,k}},P_{\Pi_{k}^*}\Big{)}\leq \int_1^\infty A_p(t)\beta_{n-k,k+1}\left(1-\frac{1}{t}\right) \frac{\rmd t}{t^2}.
\]
If furthermore $F\in D(G_\gamma)$ and $k=k(n)$ is an intermediate sequence, then 
\[
W_p^{(2)}\Big{(}P_{\Pi_{n,k}},P_{\Pi_{k}^*}\Big{)}\to 0\quad \mbox{as $n\to\infty$}.
\]
\end{corollary}

\subsubsection{The Hill estimator and Weissman quantile estimates}
Using the previous results, we analyze the Hill estimator \citep{H75}
\[
\hat\gamma_{n,k} =\frac{1}{k}\sum_{i=1}^k \log(X_{n+1-i,n}/X_{n-k,n})
\]
and compare it to the corresponding estimator in the limit Pareto model
\[
\hat\gamma^*_{k} =\frac{1}{k}\sum_{i=1}^l \log X_{i}^*.
\]
The random variables $\log(X_i^*)$ being i.i.d. with mean $\gamma$ and variance $\gamma^2$, the central limit theorem implies $\sqrt{k}(\hat\gamma_k^*-\gamma)\stackrel{d}\longrightarrow \mathcal{N}(0,\gamma^2)$ as $n\to\infty$. Using our bound on $\mathcal{W}_p^{(2)}(P_{\Pi_{n,k}\mid Z_{n-k,n}},P_{\Pi_k^*})$, we deduce the asymptotic normality of the Hill estimator $\hat\gamma_{n,k}$ and we even obtain a non asymptotic quantitative estimate for the Wasserstein distance to the Gaussian distribution.

\begin{corollary}\label{cor:Hill_AN}
Let $1\leq p<\infty$ and $t_0\geq 1$ such that $U(t_0)>0$. For $1\leq k\leq n$ and $t\geq t_0$,
\[
W_p\Big{(}P_{\sqrt{k}(\hat\gamma_{n,k}-\gamma)\mid Z_{n-k,n}=t},\mathcal{N}(0,\gamma^2)\Big{)}\leq \sqrt{k} A_p(t)+\frac{(4+3\sqrt{2/\pi})\gamma}{\sqrt k}.
\]
\end{corollary}
Corollary~\ref{cor:Hill_AN} implies that the Hill estimator is asymptotically normal and asymptotically independent of the threshold $X_{n-k,n}$ as soon as $k\to\infty$ and $\sqrt k A_p(Z_{n-k,n})\to 0$ in probability. Under the second order condition \eqref{eq:second-order-condition} and with Proposition~\ref{prop:cv-pareto-W} at hand we retrieve the classical condition $\sqrt{k}A(n/k)\to 0$ for asymptotic normality without bias.

\medskip

For $\alpha\in (0,1)$, let $q(\alpha)=F^\leftarrow(1-\alpha)$ be the quantile of order $1-\alpha$. In order to estimate an extreme quantile $q(\alpha_n)$ with $\alpha_n=O(1/n)$, \cite{W78} proposed to extrapolate from the intermediate quantile $q(k/n)$ thanks to regular variation: 
\[
q(\alpha_n)=q(k/n)\frac{U(1/\alpha_n)}{U(n/k)}\approx q(k/n)\left(\frac{n}{k\alpha_n} \right)^\gamma.
\]
This leads to Weissman extreme quantile estimate
\[
\hat q(\alpha_n)=X_{n-k,n}\left(\frac{k}{n\alpha_n} \right)^{\hat\gamma_{n,k}}.
\]
\begin{corollary}\label{cor:Weissman_AN}
Let $1\leq p<\infty$. Assume the second order condition~\eqref{eq:second-order-condition} holds with $\rho<0$. Let $k=k(n)$ be an intermediate sequence and $\alpha_n=o(k/n)$. Then, as $n\to\infty$,
\[
W_p\Big(P_{v_n^{-1}\log\frac{\hat q(\alpha_n)}{q(\alpha_n)}\mid Z_{n-k,n}},\mathcal{N}(0,\gamma^2)\Big)=O_P\Big(\sqrt{k}A(n/k)+1/\sqrt{k}\Big)
\]
where $v_n=\log(k/(n\alpha_n))/\sqrt{k}$. 
\end{corollary}
When $\sqrt{k}A(n/k)\to 0$, we deduce the asymptotic normality
\[
v_n^{-1}\log\frac{ \hat q(\alpha_n)}{ q(\alpha_n)}\stackrel{d}\longrightarrow \mathcal{N}(0,\gamma^2).
\]
Setting $z_{1-u/2}$ the Gaussian quantile of order $1-u/2$, we obtain 
\[
\lim_{n\to\infty}\mathbb{P}\Big(q(\alpha_n)/\hat q(\alpha_n) \in \left[ e^{-z_{1-u/2} \hat\gamma_{n,k}v_n}, e^{z_{1-u/2} \hat\gamma_{n,k}v_n}\right]\Big)=1-u.
\]
This confidence interval for $q(\alpha_n)$ with asymptotic level $1-u$ is valid for any sequence $\alpha_n=o(k/n)$ but it is accurate only if $v_n\to 0$, or equivalently $\log(n\alpha_n)/\sqrt{k}\to 0$. We retrieve the fact that the Weissman extreme quantile estimator is consistent if and only if $\log(n\alpha_n)/\sqrt{k}\to 0$ . This excludes arbitrary small value of $\alpha_n$ and provides a limit for extrapolation. When $\log(n\alpha_n)/\sqrt{k}\to 0$, the delta-method can be used to recover the standard asymptotic normality 
\[
v_n^{-1}\Big( \frac{\hat q(\alpha_n)}{q(\alpha_n)}-1\Big)\stackrel{d}\longrightarrow \mathcal{N}(0,\gamma^2),
\]
see Theorem 4.3.8 and Corollary 4.3.9 in \citet{dHF06}.

\subsection{Extensions}

\subsubsection{Dealing with the bias}
The asymptotic normality of the Hill estimator is classically considered under the second order condition \eqref{eq:second-order-condition} with the intermediate sequence $k=k(n)$ satisfying $\sqrt{k}A(n/k)\to \lambda\in\mathbb{R}$. The next theorem extends the previous results to cover this case as well. We define the empirical measure 
\[
\Pi_{k,t}^*=\frac{1}{k}\sum_{i=1}^k \varepsilon_{X_i^*\big(1+ A(t)\frac{(X_i^*)^{\rho/\gamma}-1}{\rho}\big)},\quad k\geq 1,t\geq 1,
\]
with $X_1^*,\ldots,X_k^*$ i.i.d. with Pareto distribution $P_\alpha$.
The empirical measure $\Pi_{n,k}$ is the same as in Theorem~\ref{thm:Wass_dist_Pi_EVT}. 

\begin{theorem}\label{thm:bias-case}
Assume $F$ satisfies the second order condition \eqref{eq:second-order-condition}. Let $p\in [1,\infty)$ and $t_0\geq 1$ such that $U(t_0)>0$. For $t\geq t_0$,
\begin{align*}
&W_p^{(2)}\Big{(}P_{\Pi_{n,k}\mid Z_{n-k,n}=t},P_{\Pi_{k,t}^*}\Big{)}\\
&\leq \int_{1}^\infty \Big|\log \frac{U(zt)}{z^\gamma U(t)}-\log \Big(1+A(t)\frac{z^\rho-1}\rho\Big)\Big|^p \frac{\rmd z}{z^2}=o(A(t)).
\end{align*}
If $k=k(n)$ is an intermediate sequence, then 
\[
W_p\Big(P_{\sqrt{k}(\hat\gamma_{n,k}-\gamma)\mid Z_{n-k,n}}, \mathcal{N}\big(\sqrt k A(n/k) b(\rho),\gamma^2\big)\Big)= O_P\big(1/\sqrt k\big)+o_P\big(\sqrt{k}A(n/k)\big)
\]
with bias given by $b(\rho)=\int_1^\infty \frac{z^\rho-1}{\rho}\frac{\rmd z}{ z^{2}}$.
\end{theorem}
Under the condition $\sqrt k A(n/k)\to\lambda$,  the asymptotic normality of the Hill estimator follows straightforwardly:
\[
\sqrt{k}(\hat\gamma_{n,k}-\gamma)\stackrel{d}\to \mathcal{N}(\lambda b(\rho),\gamma^2),
\]
with $\hat\gamma_{n,k}$ asymptotically independent of the threshold $X_{n-k,n}$.

\subsubsection{Analysis of the case $\gamma\in\mathbb{R}$ and probability weighted moments}
We consider the theory for all domain of attractions, i.e. without the simplifying assumption $\gamma>0$. Let $X_1,\ldots,X_n$ be an i.i.d. sample with distribution $F\in D(G_\gamma)$, $\gamma\in\mathbb{R}$ and $X_1^*,\ldots,X_k^*$ be i.i.d. with GP distribution $H_\gamma$. In view of Equation~\eqref{eq:cv-GPD2}, we compare the empirical distributions
\[
\Pi_{n,k}=\sum_{i=1}^k \varepsilon_{\tilde X_{n+1-i,n}} \quad \mbox{and}
\quad \Pi_k^*=\sum_{i=1}^n \varepsilon_{X_i^*}
\]
where $\tilde X_{n+1-i,n}=(X_{n+1-i,n}-X_{n-k,n})/f(X_{n-k,n})$. 
We focus on the case $\gamma<1$ to ensure that the GP distribution $H_\gamma$ defined by Equation~\eqref{eq:def-GPD} has a finite mean. Then $H_\gamma\in\mathcal{W}_p([0,\infty))$ for all $p\in[1, 1/\gamma_+)$ with $\gamma_+=\max(0,\gamma)$ and the convention $1/0=\infty$. When focusing on the PWM estimator of the extreme value index, the restriction $\gamma<1$ is sensible because it is required for consistency. We define the function
\[
 A_p'(t)=\Big(\int_1^\infty \Big| \frac{U(zt)-U(t)}{a(t)}-\frac{z^\gamma -1}{\gamma} \Big|^p \frac{\rmd z}{z^2}\Big)^{1/p},\quad t\geq 1.
\]

\begin{theorem}\label{thm:cv-gpd-W}
Assume $F\in D(G_\gamma)$ with $\gamma<1$ and let $p\in[1, 1/\gamma_+)$. 
\begin{enumerate}
\item[i)] (convergence to the GP distribution) In the Wasserstein space $\mathcal{W}_p([0,\infty))$, 
\[
W_p(P_{(X-U(t))/a(t)\mid Z>t},H_\gamma)= A_p'(t)\to 0\quad \mbox{as $t\to\infty$}.
\]
Under the second order condition~\eqref{eq:second-order-condition2}, 
\[
A_p'(t)\sim c'_p(\gamma,\rho) A(t),\quad \mbox{as $t\to\infty$,}
\]
with $c_p'(\gamma,\rho)=\Big(\int_1^\infty \Psi_{\gamma,\rho}(x)^px^{-2}\rmd x\Big)^{1/p}$.
\item[ii)] (PoT method) In the second order Wasserstein space $\mathcal{W}_p^{(2)}([0,\infty))$,
\[
W_p^{(2)}\Big{(}P_{\Pi_{n,k} \mid Z_{n-k,n}=t},P_{\Pi_k^*}\Big{)}= A_p'(t).
\]
If $k=k(n)$ is an intermediate sequence and the second order condition \eqref{eq:second-order-condition} holds, then
\[
W_p^{(2)}\Big{(}P_{\Pi_{n,k} \mid Z_{n-k,n}},P_{\Pi_k^*}\Big{)}=c_p'(\gamma,\rho)A(n/k)(1+o_P(1)).
\]
\end{enumerate}
\end{theorem}

Following \cite{HW87} (see also \cite{BGTS04} Chapter 5.3), 
the PWM estimators
\[
\hat M_{k,n,s}=\frac{1}{k} \sum_{i=1}^k \Big(1-\frac{i}{k} \Big)^s\tilde X_{n+1-i,n},\quad s=0,1,
\]
are used to estimate the GP parameters after a suitable transformation. We compare the PWM estimators with those of the limit GP model
\[
\hat M_{k,s}^*=\frac{1}{k} \sum_{i=1}^k \Big(1-\frac{i}{k} \Big)^s X_{i,k}^*,\quad s=0,1,
\]
with $X_{i,k}^*\leq\cdots\leq X_{k,k}^*$ the order statistics of the GP sample.
The estimation of $\gamma$ relies on the following relations: for $H\rightsquigarrow H_{\gamma,\sigma}$ having a GP distribution with shape $\gamma<1$ and scale $\sigma>0$, it holds
\[
m_0:=\mathbb{E}(H)=\frac{\sigma}{(1-\gamma)} \quad\mbox{and}\quad m_1:=\mathbb{E}(H(1-F_H(H)))=\frac{\sigma}{2(2-\gamma)}
\] 
or equivalently
\[
\gamma=2-\frac{m_0}{m_0-2m_1}\quad\mbox{and}\quad \sigma=\frac{2 m_0m_1}{m_0-2m_1}.
\]
This suggests the estimators  
\begin{equation}\label{eq:def-PWM}
(\hat{\gamma}_{n,k}^{PWM},\hat{\sigma}_{n,k}^{PWM})=g(\hat M_{k,n,0},\hat M_{k,n,1})
\end{equation}
where $g(x,y)=\left(2-x/(x-2y),2xy/(x-2y)\right)$. Note that in our framework the probability weighted moment are computed for the normalized exeedances $\tilde X_{n+1-i,n}= (X_{n+1-i,n}-X_{n-k,n})/f(X_{n-k,n})$ and not directly for the exceedances $X_{n+1-i,n}-X_{n-k,n}$, $i=1,\ldots,k$. This has no effect on the estimation of the extreme value index $\gamma$ but it normalizes the estimator of the scale $\sigma$.
\begin{corollary}\label{cor:AN-pwm-estimators}
In the Wasserstein space $\mathcal{W}_p(\mathbb{R}^2)$,
\[
W_p\Big{(}P_{(\hat{M}_{k,n,s})_{ s=0,1} \mid Z_{n-k,n}=t},P_{(\hat{M}_{k,s}^*\Big{)}_{s=0,1}})\leq A_p'(t).
\]
As a consequence, if $F \in D(G_\gamma)$ with $\gamma < 1/2$ satisfies the second order condition~\eqref{eq:second-order-condition2} and $k=k(n)$ is an intermediate sequence such that $\sqrt{k}A(n/k)\to 0$, then 
\[
\sqrt{k}\Big{(}\hat{M}_{k,n,0}-m_0,\hat{M}_{k,n,1}-m_1\Big{)}\overset{d}{\rightarrow}\mathcal{N}(0,\Gamma)
\] 
and
\[
\left(\sqrt{k}(\hat{\gamma}^{PWM}_{n,k}-\gamma),\sqrt{k}(\hat{\sigma}_{n,k}^{PWM}-1)\right)\overset{d}{\rightarrow}\mathcal{N}(0,\Sigma)
\]
 as $n\rightarrow \infty$, with the covariance matrices $\Gamma$ and $\Sigma$ respectively given by \citet[Equation (5.3) p.28]{H86} and \citet[expression (iii) p. 162]{BGTS04}.
\end{corollary}

\subsubsection{Analysis of the block maxima method} 
We discuss shortly the block maxima method to demonstrate that our results for the PoT method can be extended to the BM case as well. We assume $F\in D(G_\gamma)$ so that the block maxima converge in distribution, that is Equation \eqref{eq:cv-GEV} holds with GEV limit distribution $G_\gamma$ given by~\eqref{eq:def-GEV}. We introduce the function $V(t)= F^\leftarrow (e^{-1/t})$, $t>0$, so that $X=V(Z)$ has distribution $F$ if $Z$ has a unit Fr\'echet distribution. It is convenient since, for $m\geq 1$, the distribution of $V(mZ)$ is equal to the distribution $F^m$ of the the block maxima with size $m$. Note that the functions $U$ and $V$ are asymptotically equivalent so that the first order condition~\eqref{eq:first-order-condition} can be formulated equivalently with $U$ replaced by $V$. However, when it comes to second order condition, the conditions for $U$ and $V$ are not equivalent, see \cite{DdHL03} and the discussion in the appendix there. The second-order condition for $V$ reads
\begin{equation}\label{eq:second-order-condition-BM}
\lim_{t\to \infty} \frac{\frac{V(tx)-V(t)}{a(t)}-\frac{x^\gamma-1}{\gamma}}{A(t)} = \Psi_{\gamma,\rho}(x),\quad x>0,
\end{equation}
with $A$ regularly varying with index $\rho\leq 0$ and $\Psi_{\gamma,\rho}$ defined in Equation~\eqref{eq:limit-second-order-condition}.
Consider $X_1,...,X_n$ an i.i.d sample with distribution $F$ and, for $m\geq 1$ and $k=\lfloor n/m\rfloor$, the block maxima
\[
M_i:=\underset{(i-1)m+1\leq j \leq im}{\max}(X_j-b_m)/a_m,\quad 1\leq i\leq k.
\] 
Consider also $M_1^*,\ldots,M_k^*$ and i.i.d sample with GEV distribution $G_\gamma$. We compare the empirical measures
\[
\Pi_{n,m}:=\frac{1}{k}\sum_{i=1}^k\varepsilon_{(M_i-b_m)/a_m} \text{ and } \Pi_k^*:=\frac{1}{k}\sum_{i=1}^k\varepsilon_{X_i^*}
\]
in the second order Wasserstein space $\mathcal{W}_p^{(2)}([0,\infty))$ and define
\begin{equation}
A_p''(t)=\left(\int_0^\infty \left|\frac{V(tz)-V(t)}{a(t)}-\frac{z^\gamma-1}{\gamma}\right|^p e^{-1/z}\frac{\mathrm{d}z}{z^2}\right)^{1/p},\quad t>0.
\end{equation}

\begin{theorem}\label{prop:cv-gev-W}
Let $\gamma <1$ and $p \in [1,1/\gamma_+)$. 
\begin{enumerate}
\item[i)] In the Wasserstein space $\mathcal{W}_p([0,\infty))$, 
\[
W_p\Big{(}P_{(M_1-b_m)/a_m},G_\gamma\Big{)}=A_p''(m),\quad m\geq 1.
\]
Consequently, in the second order Wasserstein space $\mathcal{W}_p^{(2)}([0,\infty))$,
\[
W_p^{(2)}\Big{(}P_{\Pi_{n,m}},P_{\Pi_k^*}\Big{)}= A_p''(m),\quad 1\leq m\leq n.
\]
\item[ii)] If $F$ has a finite left endpoint and satisfies the first order condition \eqref{eq:first-order-condition} and if $m=m(n)\to\infty$, then
\[
W_p^{(2)}\Big{(}P_{\Pi_{n,m}},P_{\Pi_k^*}\Big{)} \rightarrow 0\quad \mbox{as $n \rightarrow \infty$}.
\]
\item[iii)] If furthermore the second order condition \eqref{eq:second-order-condition-BM} is satisfied, then
\[
W_p^{(2)}\Big{(}P_{\Pi_{n,m}},P_{\Pi_k^*}\Big{)} = c_p''(\gamma,\rho)A(m)(1+o_P(1)),\quad \mbox{as $n \rightarrow \infty$},
\] 
with $c_p''(\gamma,\rho)=\left(\int_0^\infty \Psi_{\gamma,\rho}(z)^p e^{-1/z}\frac{\mathrm{d}z}{z^2}\right)^{1/p}$.
\end{enumerate}
\end{theorem}
Probability weighted moments in the BM method can then be analyzed similarly as in the PoT method with similar results as in Corollary~\ref{cor:AN-pwm-estimators}.

\section{Proofs}
\subsection{Proofs related to Section~\ref{sec:2}}

\begin{proof}[Proof of Theorem~\ref{theo:wasserstein-emp-measure}]
The case $W_p(P,P^*)=0$ is trivial because we have then $P=P^*$ so that the two samples have the same distribution $P_{\Pi_n}=P_{\Pi_n^*}$ and $W_p^{(2)}( P_{\Pi_n},P_{\Pi_n^*})=0$. We assume in the sequel that $W_p(P,P^*)>0$. 

\smallskip\noindent
\textit{Proof of the upper bound $W_p^{(2)}( P_{\Pi_n},P_{\Pi_n^*})\leq W_p(P,P^*)$.}\\
For the sake of generality, we consider $p\in [1,\infty]$ and do not assume the metric space $(\mathcal{X},d)$ complete and separable. Let $\varepsilon >0$ and $(\tilde{X},\tilde{X}^*)$ be a coupling between $P$ and $P^*$ such that
\[
\| d(\tilde{X},\tilde{X}^*)\|_{L^p} \leq (1+\varepsilon)W_p(P,P^*).
\]
Taking i.i.d copies $(\tilde{X}_1,\tilde{X}_1^*),...,(\tilde{X}_n,\tilde{X}_n^*)$ of $(\tilde{X},\tilde{X}^*)$, we obtain the coupling
\[
\tilde{\Pi}_n=\frac{1}{n}\sum_{i=1}^n \varepsilon_{\tilde{X}_i}\overset{d}{=} \Pi_n
\quad\mbox{and}\quad
\tilde{\Pi}_n^*=\frac{1}{n}\sum_{i=1}^n \varepsilon_{\tilde{X}_i^*}\overset{d}{=} \Pi_n^*.
\]
For fixed $\omega$, define $(Z,Z^*)=(\tilde{X}_\iota(\omega),\tilde{X}^*_\iota(\omega))$ with $\iota$ uniformly distributed on $\{1,\ldots,n\}$. Clearly $Z\rightsquigarrow \tilde{\Pi}_n(\omega)$ and $Z^*\rightsquigarrow \tilde{\Pi}_n^*(\omega)$ so that
\begin{equation}\label{eq:wem1}
W_p(\tilde{\Pi}_n(\omega),\tilde{\Pi}_n^*(\omega))\leq\|d(Z,Z^*)\|_{L^p}.
\end{equation}
In the case $1\leq p<\infty$, Equation~\eqref{eq:wem1} yields
\[
W_p(\tilde{\Pi}_n(\omega),\tilde{\Pi}_n^*(\omega))\leq\|d(Z,Z^*)\|_{L^p}=\left( \frac{1}{n}\sum_{i=1}^n d(\tilde{X}_i\tilde{X}_i^*)^p\right)^{1/p}
\]
and we deduce
\begin{align*}
W_p^{(2)}(P_{\Pi_n},P_{\Pi_n^*})&\leq \left[\mathbb{E}(W_p(\tilde{\Pi}_n,\tilde{\Pi}_n^*)^p)\right]^{1/p} \\
&\leq\left[\mathbb{E}\left( \frac{1}{n}\sum_{i=1}^n d(\tilde{X}_i,\tilde{X}_i^*)^p\right)\right]^{1/p} \\
&\leq \left[\mathbb{E}( d(\tilde{X}_i,\tilde{X}_i^*)^p)\right]^{1/p} \\
&\leq (1+\varepsilon)W_p(P,P^*).
\end{align*}
In the case $p =+\infty$, Equation~\eqref{eq:wem1} yields similarly
\[
W_\infty(\tilde{\Pi}_n(\omega),\tilde{\Pi}_n^*(\omega))\leq\|d(Z,Z^*)\|_{L^\infty}=\max_{1 \leq i \leq n}d(\tilde{X}_i(\omega),\tilde{X}_i^*(\omega))
\]
and we get
\begin{align*}
W_\infty(P_{\Pi_n},P_{\Pi_n^*})&\leq \|W_\infty(\tilde{\Pi}_n,\tilde{\Pi}_n^*)\|_{L^\infty} \\
&\leq \Big\|\max_{1 \leq i \leq n}d(\tilde{X}_i,\tilde{X}_i^*)\Big\|_{L^\infty} \\
 &\leq \|d(\tilde{X},\tilde{X}^*)\|_{L^\infty} \\
&\leq (1+\varepsilon)W_\infty(P,P^*).
\end{align*}
In both cases, we obtain the announced upper bound by letting $\varepsilon\to 0$.
 
\smallskip\noindent
\textit{Proof of the lower bound $W_p^{(2)}( P_{\Pi_n},P_{\Pi_n^*})\geq W_p(P,P^*)$.}\\
We assume now the metric space $(\mathcal{X},d)$ complete and separable and $p\in [1,\infty)$. We first consider the case $p=1$. By the Kantorovich-Rubinstein duality \citep[Remark 6.5]{V09},
\[
W_1(P,P^*)=\inf \int_{\mathcal{X}} \varphi \, \mathrm{d}(P-P^*)
\]
with the infimum taken over all 1-Lipschitz functions $\varphi : \mathcal{X} \mapsto \mathbb{R}$. Hence, for all $\varepsilon>0$, there exists a 1-Lipschitz function $\varphi : \mathcal{X} \mapsto \mathbb{R}$ such that 
\begin{equation}\label{eq:wem3}
\int_{\mathcal{X}} \varphi \,\mathrm{d}(P-P^*) \geq (1-\varepsilon)W_1(P, P^*).
\end{equation}
Define the map $\varphi^\sharp: \mathcal{W}_1(\mathcal{X})\mapsto\mathbb{R}$ by 
$\varphi^\sharp(\pi)=\int_{\mathcal{X}}\varphi \,\mathrm{d}\pi$.
The integral is well defined because $\varphi$ is Lipschitz and $\pi\in \mathcal{W}_1(\mathcal{X})$. We prove below that $\varphi^\sharp$ is 1-Lipschitz: for $\pi_1,\pi_2 \in \mathcal{W}_1(\mathcal{X})$ and $(Z_1,Z_2)$ a coupling between $\pi_1$ and $\pi_2$, 
\[
\varphi^\sharp(\pi_1)-\varphi^\sharp(\pi_2)
= \int_{\mathcal{X}} \varphi\,\mathrm{d}(\pi_1-\pi_2)=\mathbb{E}[\varphi(Z_1)-\varphi(Z_2)]
\]
whence, since $\varphi$ is $1$-Lipschitz,
\[
|\varphi^\sharp(\pi_1)-\varphi^\sharp(\pi_2)|
\leq \mathbb{E}[\left|\varphi(Z_1)-\varphi(Z_2)\right|] \leq \mathbb{E}[d(Z_1,Z_2)].
\]
Taking the infimum over all couplings $(Z_1,Z_2)$, we get
\[
|\varphi^\sharp(\pi_1)-\varphi^\sharp(\pi_2)| \leq W_1(\pi_1,\pi_2),
\]
proving that $\varphi^\sharp$ is $1$-Lipschitz. 
Using the Kantorovich-Rubinstein duality again, we deduce
\begin{equation}\label{eq:wem4}
W_1^{(2)}\Big{(} P_{\Pi_n},P_{\Pi_n^*}\Big{)}\geq \int_{\mathcal{W}_1(\mathcal{X})}\varphi^\sharp \,\mathrm{d}(P_{\Pi_n}-P_{\Pi_n^*}),
\end{equation}
where the right hand side equals
\begin{align}
\int_{\mathcal{W}_1(\mathcal{X})}\varphi^\sharp \,\mathrm{d}(P_{\Pi_n}-P_{\Pi_n^*})
&= \mathbb{E}[\varphi^\sharp(\Pi_n)-\varphi^\sharp(\Pi_n^*)] \nonumber\\
&=\mathbb{E}\Big[\frac{1}{n}\sum_{i=1}^n \varphi(X_i)-\frac{1}{n}\sum_{i=1}^n \varphi(X_i^*)\Big] \nonumber \\
&=\mathbb{E}\left[\varphi(X)- \varphi(X^*)\right] \nonumber \\
&=\int_{\mathcal{X}} \varphi \,\mathrm{d}(P-P^*)\label{eq:wem5}. 
\end{align}
Equations \eqref{eq:wem3}, \eqref{eq:wem4} and \eqref{eq:wem5} together entail
\[
W_1^{(2)}( P_{\Pi_n},P_{\Pi_n^*})\geq (1-\varepsilon)W_1(P,P^*).
\]
We obtain the announced lower bound by letting $\varepsilon \to 0$.

We now consider the case $p \in (1,+\infty)$. By the Kantorovich duality \citep[Theorem 5.10(i)]{V09}, for all $\varepsilon>0$ there exist functions $\varphi\in L^1(\mathcal{X},P)$ and $\psi \in L^1(\mathcal{X},P^*)$ such that 
\[
\varphi(x)-\psi(y) \leq d^p(x,y), \quad x,y \in \mathcal{X},
\]
and
\[
\int_\mathcal{X}\varphi \,\mathrm{d}P-\int_\mathcal{X} \psi \,\mathrm{d}P^* \geq (1-\varepsilon)W_p^p(P,P^*).
\]
Define the functions $\varphi^\sharp,\psi^\sharp:\mathcal{W}_p(\mathcal{X})\to\mathbb{R}$ by $\varphi^\sharp(\pi)=\int_\mathcal{X}\varphi \,\mathrm{d}\pi$ and $\psi^\sharp(\pi)=\int_\mathcal{X} \psi \,\mathrm{d}\pi$. We have $\varphi^\sharp\in L^1(\mathcal{W}(\mathcal{X}),P_{\Pi_n})$ and
\[
\int_{\mathcal{W}(\mathcal{X})}\varphi^\sharp \mathrm{d}P_{\Pi_n}= \mathbb{E}[\varphi^\sharp(\Pi_n)]=\mathbb{E}\Big[\frac{1}{n}\sum_{i=1}^n \varphi(X_i)\Big]=\mathbb{E}[\varphi(X)]=\int_{\mathcal{X}}\varphi \mathrm{d}P.
\]
Similarly, $\psi^\sharp\in L^1(\mathcal{W}(\mathcal{X}),P_{\Pi_n^*})$. Furthermore, for $\pi_1,\pi_2 \in \mathcal{W}_p(\mathcal{X})$ and $(Z_1,Z_2)$ a coupling between $\pi_1$ and $\pi_2$, we have 
\[
\varphi^\sharp(\pi_1)-\psi^\sharp(\pi_2)                 =\mathbb{E}[\varphi(Z_1)-\psi(Z_2)]\leq\mathbb{E}[d^p(Z_1,Z_2)], 
\]
and, taking the infimum over all couplings,
\[
\varphi^\sharp(\pi_1)-\psi^\sharp(\pi_2) \leq W_p^p(\pi_1,\pi_2).
\]
We deduce, by the Kantorovich duality,
\begin{align*}
W_p^{(2)}(P_{\Pi_n},P_{\Pi_n^*})^p&\geq \int_{\mathcal{W}_p(X)}\varphi^\sharp \,\mathrm{d}P_{\Pi_n}-\int_{\mathcal{W}_p(X)} \psi^\sharp \,\mathrm{d}P_{\Pi_n^*} \\
&=\int_\mathcal{X}\varphi \,\mathrm{d}P-\int_\mathcal{X} \psi \,\mathrm{d}P^* \\
&\geq (1-\varepsilon)W_p^p(P,P^*).
\end{align*}
Letting $\varepsilon\to 0$, the announced lower bound follows.
\end{proof}

\begin{proof}[Proof of Corollary~\ref{cor:lipschitz-statistic}]
Let $(\tilde{\Pi}_n,\tilde{\Pi}_n^*)$ be a coupling between $P_{\Pi_n}$ and $P_{\Pi_n}$. Since $(S(\tilde{\Pi}_n),S(\tilde{\Pi}_n^*))$ is a coupling between $P_{S_n}$ and $P_{S_n^*}$, 
\begin{align*}
W_p(P_{S_n},P_{S_n^*}) & \leq \mathbb{E}[W_p(S(\tilde{\Pi}_n),S(\tilde{\Pi}_n^*))^p]^{1/p} \\
            & \leq \mathrm{Lip}(S) \mathbb{E}[ W_p(\tilde{\Pi}_n,\tilde{\Pi}_n^*)^p]^{1/p}.
\end{align*}
Taking the infimum in the right hand side over all couplings between $P_{\Pi_n}$ and $P_{\Pi_n^*}$, we get 
\[
W_p(P_{S_n},P_{S_n^*}) \leq \mathrm{Lip}(S)W_p^{(2)}(P_{\Pi_n},P_{\Pi_n^*}).
\]
Theorem \ref{theo:wasserstein-emp-measure} finally entails $W_p(P_{S_n},P_{S_n^*}) \leq \mathrm{Lip}(S)W_p(P,P^*)$.
\end{proof}

\begin{proof}[Proof of Proposition~\ref{prop:lipschitz-statistic}]
Proof of point i): let $\pi_1$ and $\pi_2 \in \mathcal{W}_p(\mathcal{X})$ and $(X_1,X_2)$ be a coupling between $\pi_1$ and $\pi_2$. By the triangle inequality, 
\begin{align*}
	|S(\pi_1)-S(\pi_2)| & = \big|\|d(x_0,X_1)\|_{L^q}-\|d(x_0,X_2)\|_{L^q}\big| \\
	          & \leq \|d(x_0,X_1)-d(x_0,X_2)\|_{L^q} \\
	          & \leq \|d(X_1,X_2)\|_{L^q}.
\end{align*}
Taking the infimum in the right hand side, we get
\[
|S(\pi_1)-S(\pi_2)| \leq W_q(\pi_1,\pi_2) \leq W_p(\pi_1,\pi_2).
\]
Moreover, for $\pi_1=\varepsilon_{x_0}$ and $\pi_2=\varepsilon_{x}$ with $x\neq x_0$, we have $S(\pi_1)=0$, $S(\pi_2)=d(x_0,x)>0$ and $W_p(\pi_1,\pi_2)=d(x_0,x)$. Hence $|S(\pi_1)-S(\pi_2)|=W_p(\pi_1,\pi_2)$ and $\mathrm{Lip}(S)=1$.

Proof of point ii): Let $\pi_1$ and $\pi_2 \in \mathcal{W}_p(\mathcal{X})$. By the triangle inequality in $L^q([0,1],(1-u)^s u^r \mathrm{d}u)$, we have 
\[
|S(\pi_1)-S(\pi_2)| \leq \Big(\int_0^1|F^\leftarrow_{\pi_1}(u)-F^\leftarrow_{\pi_2}(u)|^q(1-u)^s u^r \mathrm{d}u\Big)^{1/q}.
\]
For all $u \in [0,1]$, we have $(1-u)^s u^r \leq 1$. We deduce, using the explicit expression of the Wasserstein distance on the real line (see Example~\ref{ex:1}),
\begin{align*}
|S(\pi_1)-S(\pi_2)| & \leq \Big(\int_0^1|F^\leftarrow_{\pi_1}(u)-F^\leftarrow_{\pi_2}(u)|^q \mathrm{d}u \Big)^{1/q}\\
& \leq W_q(\pi_1,\pi_2)\leq W_p(\pi_1,\pi_2).
\end{align*}
\end{proof}

\subsection{Proofs related to Section~\ref{sec:3}}
\begin{proof}[Proof of Proposition~\ref{prop:cv-pareto-W}]
\textit{Proof of $i)$.} We begin with the proof of Equation~\eqref{eq:prop-cv-pareto-W} by giving the expression of the Wasserstein distance on $\mathcal{X}=[1,\infty)$ with logarithmic distance $d$ given by \eqref{eq:log-dist}. Because the logarithm is an isometry between $([1,\infty),d)$ and $(\mathbb{R},d_2)$ with $d_2$ the usual Euclidean distance, the Wasserstein distance between probability measures $P_1$ and $P_2$ on $(\mathcal{X},d)$ is equal to the Wasserstein distance between the image measures $P_1\circ \log^{-1}$ and $P_2\circ \log^{-1}$ on $(\mathbb{R},d_2)$. Using the explicit form of the Wasserstein distance on $\mathbb{R}$ recalled in Example~\ref{ex:1}, we deduce
\[
W_p(P_1,P_2)=\Big( \int_{0}^1 |\log F_1^\leftarrow(u)-\log F_2^\leftarrow(u)|^p\,\mathrm{d}u\Big)^{1/p},
\]
with $F_1^\leftarrow$ and $F_2^\leftarrow$ the quantile functions of $P_1$ and $P_2$ respectively. Introducing the change of variable $ z=1/(1-u)$ and the tail quantile functions $U_i(t)=F_i^\leftarrow(1-1/t)$, $i=1,2$, we get
\begin{equation}\label{eq:Wd}
W_p(P_1,P_2)=\left( \int_{1}^\infty \left|\log \frac{U_1(z)}{U_2(z)}\right|^p\,\frac{\mathrm{d}z}{z^2}\right)^{1/p}.
\end{equation}
Similarly, in the case $p=\infty$, we obtain
\[
W_\infty(P_1,P_2)=\sup_{u\in (0,1)} \left|\log F_1^\leftarrow(u)-\log F_2^\leftarrow(u)\right|=\sup_{z>1} \left|\log \frac{U_1(z)}{U_2(z)}\right|.
\]
Specializing these formulas when $P_1=P_{U(t)^{-1}X\mid Z>t}$ and $P_2=P_\alpha$, we obtain Equation~\eqref{eq:prop-cv-pareto-W} because 
\[
U_1(z)=\frac{U(tz)}{U(t)} \quad\mbox{and}\quad U_2(z)=z^\gamma.
\]
To check the first equality, one can use the fact that, given $Z>t$, $Z$ has the same distribution as $tZ$ so that $X/U(t)=U(Z)/U(t)$ has the same distribution as $U(tZ)/U(t)$.\\
\textit{Proof of $ii)$.} We next consider the properties of $A_p$ when $F\in D(G_\gamma)$, or equivalently, under the first order condition $\eqref{eq:RV}$. Since the function $U$ is regularly varying at infinity with order $\gamma>0$, Potter's bounds \citep[Proposition B.1.9 p. 366]{dHF06} imply that, for all $\delta > 0$, there exists $t_1 > 1$ such that 
\begin{equation}\label{eq:Pottersbound}
(1-\delta)z^{-\delta} \leq\frac{U(tz)}{z^\gamma U(t)} \leq (1+\delta)z^{\delta}\quad \mbox{for all $t\geq t_1$ and $z>1$}.
\end{equation}
This implies that, for $p\in [1,\infty)$, $A_p(t)$ defined by Equation~\eqref{eq:prop-cv-pareto-W} is bounded for $t\geq t_1$. On the other hand, the monotonicity of $U$ implies that, for $t\in[t_0,t_1]$, $A_p^p(t)\leq \int_{1}^\infty \max \left(\big|\log \frac{U(zt_1)}{z^\gamma U(t_0)}\big|^p,\big|\log \frac{U(zt_0)}{z^\gamma U(t_1)}\big|^p\right) \frac{\rmd z}{z^2}$ , whence $A_p$ is bounded on $[t_0,\infty)$.
Furthermore, the first order condition~\eqref{eq:RV} implies
\[
\lim_{t\to\infty}\Big|\log \frac{U(zt)}{z^\gamma U(t)}\Big|^p =0 \quad \mbox{for all $z>1$}.
\]
Applying the dominated convergence theorem, with the dominating function provided by Potter's bound \eqref{eq:Pottersbound}, we deduce 
$\lim_{t\to\infty} A_p(t)=0$.\\
\textit{Proof of $iii)$.} We finally consider the asymptotic behavior of $A_p$ under the second order condition~\eqref{eq:second-order-condition}. With the auxiliary function
$f(z)=\log\left(U(z)/z^\gamma\right)$, $z>1$, 
the ratio $A_p(t)/A(t)$ can be rewritten as 
\begin{equation}\label{eq:ratio}
\frac{A_p(t)}{A(t)}= \left\{\begin{array}{ll}
 \left(\int_{1}^\infty \left|\frac{f(tz)-f(t)}{A(t)}\right|^p \frac{\rmd z}{z^2}\right)^{1/p}& \mbox{for $1\leq p < \infty$,}\\
 \sup_{z>1} \left|\frac{f(tz)-f(t)}{A(t)}\right| & \mbox{for $p=\infty$.} 
\end{array}\right.
\end{equation}
From the second order condition \eqref{eq:second-order-condition}, the function $f$ satisfies, as $t\to\infty$, 
\begin{equation}\label{eq:ERV}
\frac{f(tz)-f(t)}{A(t)}=\frac{\log\left(\frac{U(tz)}{z^\gamma U(t)}\right)}{A(t)}\sim\frac{\frac{U(tz)}{z^\gamma U(t)}-1}{A(t)} \longrightarrow \frac{z^\rho -1}{\rho}.
\end{equation}
In the terminology of \citet[Definition B.2.3]{dHF06}, the function $f$ is of extended regular variations at infinity with index $\rho\leq 0$.
Equations~\eqref{eq:ratio} and \eqref{eq:ERV} together suggest 
\begin{equation}\label{eq:lim_ApsurA}
\frac{A_p(t)}{A(t)}\longrightarrow \left\{\begin{array}{ll}
 \left(\int_{1}^\infty \left|\frac{z^\rho -1}{\rho}\right|^p \frac{\rmd z}{z^2}\right)^{1/p}& \mbox{for $1\leq p < \infty$,}\\
 \sup_{z>1} \left|\frac{z^\rho -1}{\rho}\right| & \mbox{for $p=\infty$.} 
\end{array}\right.
\end{equation}
To justify the limits in Equation~\eqref{eq:lim_ApsurA}, we use Theorem B.2.18 p.383 in \cite{dHF06}: possibly replacing $A$ by an asymptotically equivalent function, we can assume that for arbitrary $\varepsilon,\delta >0$, there exists $t_0$ such that
\begin{equation}\label{eq:Drees}
\left| \frac{f(tz)-f(t)}{A(t)}-\frac{z^\rho -1}{\rho}\right| \leq\varepsilon z^{\rho+\delta},\quad \mbox{for all $t\geq t_0$, $z>1$}.
\end{equation}
In the case $p=[1,\infty)$, the limit \eqref{eq:lim_ApsurA} is a consequence of the dominated convergence Theorem where the pointwise convergence is given by Equation~\eqref{eq:ERV} and the dominating function by Equation~\eqref{eq:Drees}. \\
In the case $p=\infty$ and $\rho<0$, the right hand side in Equation~\eqref{eq:Drees} can be made arbitrary small uniformly in $z>1$ by choosing $\varepsilon$ and $\delta$ small enough (note that $z^{\delta+\rho}<1$ for $\delta+\rho<0$ and $z>1$). This is enough to justify the limit \eqref{eq:lim_ApsurA} and the value of the limit is $1/|\rho|$.\\
In the case $p=\infty$ and $\rho=0$, Equation~\eqref{eq:Drees} implies
\[
\frac{f(tz)-f(t)}{A(t)}\geq \log(z)-\varepsilon z^{\delta}
\]
For arbitrary $M>0$, one can choose $z>1$ large enough and $\varepsilon,\delta>0$ small enough so that the right hand side is larger than $M$ for $t\geq t_0$. Then, for $t\geq t_0$, 
\[
\frac{A_\infty(t)}{A(t)}= \sup_{z>1} \left|\frac{f(tz)-f(t)}{A(t)}\right| \geq M, 
\]
proving $\lim_{t\to\infty}A_\infty(t)/A(t)=+\infty$.
\end{proof}

\begin{proof}[Proof of Theorem~\ref{thm:Wass_dist_Pi_EVT}] \textit{Proof of $i)$.} For a random variable $Z$ with standard unit Pareto distribution, the conditional distribution of $Z$ given $Z>t$ is equal to the distribution of $tZ$. We deduce that the conditional distribution of the $k$ top order statistics $(Z_{n-k+1,n},\ldots,Z_{n,n})$ given $Z_{n-k,n}=t$ is equal to the distribution of $(t\tilde Z_{1,k},\ldots,t\tilde Z_{k,k})$ with $\tilde Z_{1,k}\leq\cdots\leq \tilde Z_{k,k} $ the order statistics of an i.i.d. sample $\tilde Z_1,\ldots,\tilde Z_k$ with standard unit Pareto margins. As a consequence, the conditional distribution of $\Pi_{n,k}$ given $Z_{n-k,n}=t$ is equal to the distribution of 
\[
\tilde\Pi_k=\frac{1}{k}\sum_{i=1}^k \varepsilon_{\tilde X_i}\quad \mbox{with $\tilde X_i=U(t\tilde Z_i)/U(t)$}.
\]
Theorem~\ref{theo:wasserstein-emp-measure} implies
\[
W_p^{(2)}\Big{(}P_{\Pi_{n,k}\mid Z_{n-k,n}=t},P_{\Pi_{k}^*}\Big{)}=W_p(P_{\tilde X},P_\alpha)=A_p(t).
\]
\textit{Proof of $ii)$.} For an intermediate sequence $k=k(n)$, $Z_{n-k,n}$ converges to $\infty$ in probability as $n\to\infty$. By Proposition~\ref{prop:cv-pareto-W}, $A_p$ has limit $0$ at infinity and it follows 
\[
W_p^{(2)}\Big{(}P_{\Pi_{n,k}\mid Z_{n-k,n}},P_{\Pi_{k}^*}\Big{)}=A_p(Z_{n-k,n})\to 0\quad \mbox{in probability.}
\]
\textit{Proof of $iii)$.} By Proposition~\ref{prop:cv-pareto-W} point $iii)$, the second order condition implies $A_p(t)\sim c_p(\rho)A(t)$ as $t\to\infty$. On the other hand, for an intermediate sequence $k=k(n)$, $Z_{n-k,n}=(n/k)(1+o_P(1))$ as $n\to\infty$. Combining the two results, we deduce 
\[
W_p^{(2)}\Big{(}P_{\Pi_{n,k}\mid Z_{n-k,n}},P_{\Pi_{k}^*}\Big{)}=A_p(Z_{n-k,n})=c_p(\rho)A(n/k)(1+o_P(1)).
\]
\end{proof}

\begin{proof}[Proof of Corollary~\ref{cor:Wass_dist_Pi_EVT}]
When $X$ is positive and bounded away from 0, the equality $W_p^{(2)}(P_{\Pi_{n,k}\mid Z_{n-k,n}=t},P_{\Pi_{k}^*})=A_p(t)$ holds for all $t\geq 1$. Integrating with respect to the distribution of $ Z_{n-k,n}$ and using the convexity of the Wasserstein distance \citep[Theorem 4.8]{V09}, we deduce
\[
W_p^{(2)}\Big{(}P_{\Pi_{n,k}},P_{\Pi_{k}^*}\Big{)}\leq \mathbb{E}[A_p(Z_{n-k,n})]=\int_1^\infty A_p(t)\beta_{n-k,k+1}\left(1-\frac{1}{t}\right) \frac{\rmd t}{t^2}.
\]
The last equality holds since $Z_{n-k,n}$ has density $\beta_{n-k,k+1}(1-1/t)/t^2$. For an intermediate sequence $k=k(n)$, $Z_{n-k,n}\to \infty$ in probability and Proposition~\ref{prop:cv-pareto-W} $ii)$ implies that $A_p(Z_{n-k,n})\to 0$ in probability. Since $A_p(t)$ is bounded, $\mathbb{E}[A_p(Z_{n-k,n})]\to 0$ whence the Wasserstein distance converges to $0$.
\end{proof}

\begin{proof}[Proof of Corollary~\ref{cor:Hill_AN}]
By the triangular inequality,
\begin{align*}
& W_p\Big{(}P_{\sqrt{k}(\hat\gamma^{Hill}_{n,k}-\gamma)\mid Z_{n-k,n}=t},\mathcal{N}(0,\gamma^2)\Big{)}\\
\leq &\ W_p\Big{(}P_{\sqrt{k}(\hat\gamma_{n,k}-\gamma)\mid Z_{n-k,n}=t},P_{\sqrt{k}(\hat{\gamma}^{*}_{k}-\gamma)}\Big{)}+W_p\Big{(}P_{\sqrt{k}(\hat{\gamma}^{*}_{k}-\gamma)},\mathcal{N}(0,\gamma^2)\Big{)}.
\end{align*}
The first term is estimated by Corollary~\ref{cor:lipschitz-statistic} with $\hat{\gamma}_{n,k}=S(\Pi_{n,k})$, $\hat{\gamma}^*_k=S(\Pi_k^*)$ and $S: \mathcal{W}_p([1,\infty))\to\mathbb{R}$ given by $\varphi(\pi)=\int_1^{+\infty}\log(x)\pi(\mathrm{d}x)$. Theorem~\ref{thm:Wass_dist_Pi_EVT} and Proposition~\ref{prop:lipschitz-statistic} stating that $S$ is Lipschitz with $\mathrm{Lip}(S)=1$ imply 
\[
W_p\Big{(}P_{\hat\gamma_{n,k}\mid Z_{n-k,n}=t},P_{\hat\gamma_{n,k}^*}\Big{)}\leq \mathrm{Lip}(S) W_p^{(2)}\Big{(}P_{\Pi_{n,k}\mid Z_{n-k,n}=t},P_{\Pi_{k}^*}\Big{)}=A_p(t).
\]
A change of variable entails
\[
W_p\Big{(}P_{\sqrt{k}(\hat\gamma_{n,k}-\gamma)\mid Z_{n-k,n}=t},P_{\sqrt{k}(\hat{\gamma}^{*}_{k}-\gamma)}\Big{)}\leq \sqrt{k}A_p(t).
\]
The second term is handled using the central limit theorem in Wasserstein distance derived from Stein's method, see \citet{R11}. We write 
\[
\sqrt{k}(\hat{\gamma}^{*}_{k}-\gamma)=\frac{\gamma}{\sqrt{k}}\sum_{i=1}^k\tilde{X}_i\quad \mbox{with $\tilde{X}_i=(\log(X_i^*)-\gamma)/\gamma$}.
\] 
An application of \citet[Theorem 3.2]{R11} yields
\begin{align*}
W_p\left(P_{k^{-1/2}\sum_{i=1}^k\tilde X_i},\mathcal{N}(0,1)\right)&
\leq k^{-3/2}\sum_{i=1}^k \mathbb{E}|\tilde{X}_i|^3+\sqrt{2/\pi} k^{-1}\left(\sum_{i=1}^k \mathbb{E}\tilde{X}_i^4\right)^{1/2}\\
&=(4+3\sqrt{2/\pi})\frac{1}{\sqrt k},
\end{align*}
where the constants $\mathbb{E}|\tilde{X}_i|^3=4$ and $\mathbb{E}\tilde{X}_i^4=9$ are easily computed since $\tilde X_i$ is a centered standard exponential random variable. With a change of variable, we deduce
\begin{equation}\label{eq:hillpareto}
W_p\Big{(}P_{\sqrt{k}(\hat{\gamma}^{*}_{k}-\gamma)},\mathcal{N}(0,\gamma^2)\Big{)} \leq (4+3\sqrt{2/\pi})\frac{\gamma}{\sqrt{k}}.
\end{equation}
\end{proof}

\begin{proof}[Proof of Corollary~\ref{cor:Weissman_AN}]
We write
\[
v_n^{-1}\log \frac{\hat q(\alpha_n)}{q(\alpha_n)}=I+II+III
\]
with 
\begin{align*}
I&=\sqrt{k}(\hat\gamma_{n,k}-\gamma),\\
II&=v_n^{-1}\log \frac{X_{n-k,n}}{U(n/k)},\\
III&=v_n^{-1}\log\frac{(k/(n\alpha_n))^\gamma U(n/k)}{U(1/\alpha_n)}.
\end{align*}
 The asymptotic behavior of the first term is given by~Corollary \ref{cor:Hill_AN} together with Proposition~\ref{prop:cv-pareto-W} $iii)$:
 \[
 W_p\Big{(}P_{\sqrt{k}(\hat\gamma_{n,k}-\gamma)|Z_{n-k,n}},\mathcal{N}(0,\gamma^2)\Big{)}=O_P\left(\sqrt{k}A(n/k)+1/\sqrt{k}\right).
 \]
 The second term can be handled seeing that $X_{n-k,n}=U(Z_{n-k,n})$ with $Z_{n-k,n}=\frac{n}{k}(1+o_P(1))$. This implies  
\begin{align*}
II&=v_n^{-1}\log\frac{U(Z_{n-k,n})}{U(n/k)}\\
&=v_n^{-1}\left(\frac{U(Z_{n-k,n})}{U(n/k)}-1 \right)(1+o_P(1)) \\
&=v_n^{-1}\Big((Z_{n-k,n}k/n)^\gamma O_P(A(n/k)) \Big) (1+o_P(1))\\
&=o_P\left(\sqrt{k} A(n/k)\right),
\end{align*} 
where the third line relies on the second order condition~\eqref{eq:second-order-condition}. \\
 The third term is deterministic and satisfies
 \begin{align*}
III&= v_n^{-1}\left(\frac{(k/(n\alpha_n))^\gamma U(k/n)}{U(1/\alpha_n)}-1\right)(1+o(1))\\
&= v_n^{-1}\left(A(n/k)\frac{(k/(n\alpha_n))^\rho-1}{\rho} \right)(1+o(1))\\
&=o_P\left(\sqrt{k} A(n/k)\right),
 \end{align*}
where the second equality relies on \cite[Theorem 2.3.9.]{dHF06} and on the fact that $k/(n\alpha_n)\to\infty$ and $\rho<0$.
\end{proof}

\begin{proof}[Proof of Corollary~\ref{cor:AN-pwm-estimators} ]
First note that $\hat{M}_{k,n,s}$ and $\hat{M}^*_{k,s}$ for $s=0,1$ are the weighted probability moments associated with the empirical measures $\Pi_{n,k}$ and $\Pi_k^*$ respectively. According to Proposition \ref{prop:lipschitz-statistic} $ii)$, the weighted probability moment functional $S$ (with $q=1$, $r=0 $ and $s=0,1$) is 1-Lipschitz so that Corollary \ref{cor:lipschitz-statistic} implies 
\[
W_p\left(P_{(\hat{M}_{k,n,s})_{ s=0,1} \mid Z_{n-k,n}=t},P_{(\hat{M}_{k,s}^*)_{s=0,1}}\right)\leq W_p^{(2)}\left(P_{\Pi_{n,k} \mid Z_{n-k,n}=t},P_{\Pi_k^*}\right).
\]
Under the assumption $\sqrt{k}A(n/k)\to 0$, Theorem~\ref{thm:cv-gpd-W} ii) entails 
\[
W_p\left(P_{\sqrt{k}(\hat{M}_{k,n,s}-m_s)_{s=0,1}\mid Z_{n-k,n}},P_{\sqrt{k}(\hat{M}_{k,s}^*-m_s)_{s=0,1}}\right)=O_P\left( \sqrt{k}A(n/k)\right) \longrightarrow 0.
\]
\citet[Equation 5.3]{H86} states the joint asymptotic normality of $\hat{M}^*_{k,0}$ and $\hat{M}^*_{k,1}$: 
\[
\sqrt{k}\left(\hat{M}^*_{k,0}-m_0,\hat{M}^*_{k,1}-m_1\right)\overset{d}{\longrightarrow}\mathcal{N}(0,\Gamma)\quad \mbox{as $k\to\infty$}.
\]
Then Lemma~1 in Supplement A implies the joint asymptotic normality of $\hat{M}_{k,n,0}$ and $\hat{M}_{k,n,1}$:
\[
\sqrt{k}\left(\hat{M}_{k,n,0}-m_0,\hat{M}_{k,n,1}-m_1\right)\overset{d}{\longrightarrow}\mathcal{N}(0,\Gamma)\quad \mbox{as $k\to\infty$}.
\]
In view of Equation~\eqref{eq:def-PWM}, the asymptotic normality of $(\hat\gamma_{n,k}^{PWM},\hat\sigma_{n,k}^{PWM})$ follows by the delta method, see e.g. van der Vaart \cite{vdV98}, Theorem 3.1 p.26). Since $g$ is infinitely differentiable on $\mathbb{R}^2\backslash\{(x,y), x=2y\}$ and $(\gamma,1)$ does belong to this set, we can apply the delta method and deduce 
\begin{align*}
&\sqrt{k}\left((\hat{\gamma}_{n,k}^{PWM},\hat{\sigma}_{n,k}^{PWM})-(\gamma,1)\right)\\
=&\sqrt{k}\left(g(\hat{M}_{k,n,0},\hat{M}_{k,n,1})-g(m_0,m_1)\right)\overset{d}{\longrightarrow}\mathrm{d}g_{(\gamma,1)}(N),
\end{align*} 
with $\mathrm{d}g_{(\gamma,1)}$ the derivative of $f$ at $(\gamma,1)$ and $N\rightsquigarrow\mathcal{N}(0,\Sigma)$. We conclude the proof by noting that $\mathrm{d}g_{(\gamma,1)}(N)\stackrel{d}=\mathcal{N}(0,\Gamma)$.
\end{proof}

\section*{Acknowledgements}
The current manuscript benefited a lot from constructive comments and fruitful suggestions after presentations of the work at an early stage of the project. In particular,  Laurens de Haan suggested Proposition~\ref{prop:cv-pareto-W} $iii)$ and its proof, Ana Ferreira suggested to consider the asymptotic regime with bias $\sqrt k A(n/k)\to \lambda$ (section 3.3.1), Chen Zhou suggested to consider the block maxima method (section 3.3.3) and St\'ephane Girard suggested to consider Weissman extreme quantile estimate (Corollary~\ref{cor:Weissman_AN}). We gratefully acknowledge their contributions to the paper.

\appendix
\section{Supplementary material}

\begin{proof}[Proof of Theorem~\ref{thm:bias-case}] Similar arguments as in the proof of Theorem~\ref{thm:Wass_dist_Pi_EVT} imply
\[
W_p^{(2)}\Big{(}P_{\Pi_{n,k}\mid Z_{n-k,n}=t},P_{\Pi_{k,t}^*}\Big{)}=W_p\Big(P_{U(t)^{-1}X\mid Z>t},P_{X^*(1+ A(t)\frac{(X^*)^{\rho/\gamma}-1}{\rho})}\Big)
\]
where $X^*\rightsquigarrow P_\alpha$. Introducing a random variable $Z$ with standard unit Pareto distribution and the coupling 
\[
\frac{U(tZ)}{U(t)}\rightsquigarrow P_{U(t)^{-1}X\mid Z>t}\quad\mbox{and}\quad Z^\gamma\Big(1+ A(t)\frac{Z^{\rho}-1}{\rho}\Big)\rightsquigarrow P_{X^*(1+ A(t)\frac{(X^*)^{\rho/\gamma}-1}{\rho})},
\]
we deduce, similarly as in Equation~\eqref{eq:Wd},
\[
W_p^{(2)}\Big{(}P_{\Pi_{n,k}\mid Z_{n-k,n}=t},P_{\Pi_{k,t}^*}\Big{)}\leq\Big(  \int_{1}^\infty \Big|\log \frac{U(zt)}{z^\gamma U(t)}-\log\Big( 1+A(t)\frac{z^\rho-1}\rho\Big)\Big|^p \frac{\rmd z}{z^2}\Big)^{1/p}.
\]
Equation~\eqref{eq:Drees} with $f(z)=U(z)/z^\gamma$ implies
\[
\Big|\log \frac{U(zt)}{ z^\gamma U(t)}-A(t)\frac{z^\rho-1}{\rho}\Big|\leq \varepsilon |A(t)| z^{\rho+\delta},
\]
and,  with  $x-2x^2\leq \log(1+x)\leq  x$ for $x\geq -1/2$, we deduce 
\[
\Big|\log \frac{U(zt)}{ z^\gamma U(t)}-\log\Big( 1+A(t)\frac{z^\rho-1}\rho\Big)\Big|\leq \varepsilon |A(t)| z^{\rho+\delta}+2A(t)^2\Big|\frac{z^\rho-1}\rho \Big|^2.
\]
Taking the $L^p$-norm and $\varepsilon>0$ being arbitrary, we see that 
\[
W_p^{(2)}\Big{(}P_{\Pi_{n,k}\mid Z_{n-k,n}=t},P_{\Pi_{k,t}^*}\Big{)}=o(A(t))\quad \mbox{as $t\to\infty$}.
\]
With the same notations as in the proof of Corollary~\ref{cor:Hill_AN}, we compare \[
\hat\gamma_{n,k}=\int_1^\infty \log(x)\Pi_{n,k}(\rmd x)\quad\mbox{and}\quad\hat\gamma_{k,t}^*=\int_1^\infty \log(x)\Pi_{k,t}^*(\rmd x).
\] 
By Corollary~\ref{cor:lipschitz-statistic},
\[
W_p\Big{(}P_{\sqrt{k}(\hat\gamma_{n,k}-\gamma)\mid Z_{n-k,n}},P_{\sqrt{k}(\hat\gamma_{k,t}^*-\gamma)}\Big{)}\leq \sqrt{k} W_p^{(2)}(P_{\Pi_{n,k}\mid Z_{n-k,n}},P_{\Pi_{k,t}^*}).
\]
On the other hand,
\[
\sqrt{k}(\hat\gamma_{k,t}^*-\gamma)=\sqrt{k}(\hat\gamma_{k}^*-\gamma)+\frac{1}{\sqrt{k}}\sum_{i=1}^k\log\Big(1+A(t)\frac{Z^\rho-1}\rho\Big).
\]
The first term is compared to $\mathcal{N}(0,\gamma^2)$ thanks to Equation~\eqref{eq:hillpareto}:
\[
W_p\Big{(}P_{\sqrt{k}(\hat\gamma_k^*-\gamma)},\mathcal{N}(0,\gamma^2)\Big{)}\leq \frac{4+3\sqrt{2/\pi}}{\sqrt{k}}.
\]
The second term is compared to  $\sqrt{k}A(t)b(\rho)$. Using the triangle inequality and $x-2x^2\leq \log(1+x)\leq x$ for $x\geq -1/2$, we get
\begin{align*}
&\Big\|\frac{1}{\sqrt k}\sum_{i=1}^k\log\Big(1+A(t)\frac{Z_i^\rho-1}\rho\Big)-\sqrt{k}A(t)b(\rho)\Big\|_{L^p}\\
&\leq \sqrt{k}A(t)\Big\|\frac{1}{k}\sum_{i=1}^k\Big(\frac{Z_i^\rho-1}\rho-b(\rho)\Big)\Big\|_{L^p}+\sqrt{k} A(t)^2\Big\|\frac{1}{k}\sum_{i=1}^k\Big|\frac{Z_i^\rho-1}\rho\Big|^2\Big\|_{L^p}\\
&=o(\sqrt{k} A(t)).
\end{align*}
The last line is a consequence of the law of large numbers in $L^p$ norm. Combining the different estimates, we get
\[
W_p\Big(P_{\sqrt{k}(\hat\gamma_{n,k}-\gamma)\mid Z_{n-k,n}=t},\mathcal{N}\big(\sqrt{k}A(t)b(\rho),\gamma^2\big)\Big)=O(1/\sqrt{k})+o(\sqrt{k}A(t))
\]
and the asymptotic normality of the Hill estimator with bias $\lambda b(\rho)$ follows if $k\to\infty$ and $\sqrt{k} A(n/k)\to \lambda$.

\end{proof}

\begin{proof}[Proof of Theorem~\ref{thm:cv-gpd-W}]
\textit{Proof of $i)$.} The structure of the proof is the same as the proof of Proposition \ref{prop:cv-pareto-W}. The tail quantile function of $P_{(X-U(t))/a(t)|T>t}$ is given by $(U(tz)-U(t))/a(t)$, $z>1$ and the tail quantile function of the GP distribution  by $(z^\gamma-1)/\gamma$. The explicit expression of the Wasserstein distance given in Example \ref{ex:1} yields 
\[
W_p\Big{(}P_{(X-U(t))/a(t)|T>t},H_\gamma\Big{)}=A_p'(t).
\]
The convergence of $A_p'$ to 0 is a consequence of the dominated convergence Theorem with the domination condition provided by Drees's Theorem (Theorem B.2.18 p.383 in \cite{dHF06}): for all $\varepsilon,\delta >0$, there exists $t' >0$ such that
\[
\left|\frac{U(tz)-U(t)}{a(t)}-\frac{z^\gamma-1}{\gamma}\right|\leq \varepsilon z^{\gamma+\delta}, \quad z >1, t>t',
\] 
where $z^{(\gamma+\delta)p}/z^{2}$ is integrable on $[1,+\infty)$ for $\gamma p <1$ and $\delta$ small enough. 

Under the second order condition \eqref{eq:second-order-condition2},
\[
 \frac{A_p'(t)}{A(t)}=\left(\int_1^\infty \left| \frac{\frac{U(zt)-U(t)}{a(t)}-\frac{z^\gamma -1}{\gamma}}{A(t)} \right|^p \frac{\rmd z}{z^2}\right)^{1/p}\longrightarrow\Big(\int_1^\infty \Psi_{\gamma,\rho}(z)^pz^{-2}\rmd z\Big)^{1/p}.
\]
The convergence is proved using  the dominated convergence theorem with dominating function provided by \cite{dHF06}, Theorem B.3.10 p.392: for $\varepsilon,\delta >0$, there exists $t'>0$ such that 
\[
\left| \frac{\frac{U(zt)-U(t)}{a(t)}-\frac{z^\gamma -1}{\gamma}}{A(t)} -\Psi_{\gamma,\rho}(z)\right| \leq \varepsilon z^{\gamma +\rho +\delta}, \quad x >1, t>t',
\]
and the function $z^{(\gamma+\rho+\delta)p}/z^2$ is integrable on $[1,+\infty)$ for $(\gamma+\rho) p <1$ and $\delta$ small enough.

\textit{Proof of $ii)$.} The conditional distribution of the $k$ top order statistics $(Z_{n-k+1,n},\ldots,Z_{n,n})$ given $Z_{n-k,n}=t$ is equal to the distribution of $(t\tilde Z_{1,k},\ldots,t\tilde Z_{k,k})$ with $\tilde Z_{1,k}\leq\cdots\leq \tilde Z_{k,k} $ the order statistics of an  i.i.d. sample $\tilde Z_1,\ldots,\tilde Z_k$ with standard unit Pareto margins. We deduce that $P_{\Pi_{n,k}\mid Z_{n-k,n}=t}$ is equal to the distribution of 
\[
\tilde\Pi_k=\frac{1}{k}\sum_{i=1}^k \varepsilon_{\tilde X_i}\quad \mbox{with $\tilde X_i=(U(t\tilde Z_i)-U(t))/a(t)$}.
\]
Theorem~\ref{theo:wasserstein-emp-measure} implies
\[
W_p^{(2)}\Big{(}P_{\Pi_{n,k}\mid Z_{n-k,n}=t},P_{\Pi_{k}^*}\Big{)} \leq W_p(P_{\tilde X},H_\gamma)=A_p'(t)
\]
because $P_{\tilde X}=P_{(X-U(t))/a(t)\mid T>t}$. 

Under the second order condition \eqref{eq:second-order-condition2}, the first point of this theorem  yields $A_p'(t)\sim c_p'(\gamma, \rho) A(t)$. Since $Z_{n-k,n}= n/k(1+o_P(1))$  and $A_p'$ is regularly varying, we deduce 
\[
W_p^{(2)}\Big{(}P_{\Pi_{n,k} \mid Z_{n-k,n}},P_{\Pi_k^*}\Big{)}=A_p'(Z_{n-k,n})= c_p'(\gamma, \rho) A(n/k)(1+o_P(1)).
\]
\end{proof}

\begin{proof}[Proof of Theorem~\ref{prop:cv-gev-W}] The proof is similar to the proof of Proposition \ref{prop:cv-pareto-W}.\\
\textit{Proof of $i)$.}   The explicit form of the Wasserstein distance on $\mathbb{R}$ recalled in Example~\ref{ex:1} implies
\[
W_p(P_1,P_2)=\Big( \int_{0}^1 | F_1^\leftarrow(u)- F_2^\leftarrow(u)|^p\,\mathrm{d}u\Big)^{1/p}
\]
for all probability measures $P_1,P_2$ on $\mathbb{R}$ with quantile functions $F_1^\leftarrow,F_2^\leftarrow$ respectively. Introducing the change of variable $z=-1/\log(u)$ and the functions $V_i(t)=F^\leftarrow_i(e^{-1/t})$, $i=1,2$, we get
\[
W_p(P_1,P_2)=\Big( \int_{0}^1 | V_1(z)- V_2(z)|^p\,e^{-1/z}\frac{\rmd z}{z^2}\Big)^{1/p}.
\]
When $P_1=P_{(M_1-b_m)/a_m}$ is the distribution of the normalized block maxima and $P_2=G_\gamma$ is the GEV distribution, we get $V_1(z)=(V(mz)-V(m))/a(m)$ and $V_2(m)=(z^\gamma-1)/\gamma$, yielding
\begin{align*}
W_p(P_{(M_1-b_m)/a_m},G_\gamma)&=\left( \int_0^\infty \left|\frac{V(mz)-V(m)}{a(m)}-\frac{z^\gamma-1}{\gamma}\right|^p e^{-1/z}\frac{\rmd z}{z^2}\right)^{1/p}\\
&=A_p''(m),\quad m\geq 1.
\end{align*}
Theorem~\ref{theo:wasserstein-emp-measure} then implies,  in the second order Wasserstein space $\mathcal{W}_p^{(2)}([0,\infty))$,
\[
W_p^{(2)}\Big{(}P_{\Pi_{n,m}},P_{\Pi_k^*}\Big{)}= A_p''(m),\quad 1\leq m\leq n.
\]
\textit{Proof of $ii)$.} We prove that if $F\in\mathcal{D}(G_\gamma)$ with $\gamma<1$ and $p\in [1,1/\gamma_+)$, then $A_p''(m)\to 0$ as $m\to\infty$. By Drees Theorem (see e.g. \cite{dHF06}, Theorem B.2.18 p.383), the first order condition implies  that for all $\varepsilon,\delta>0$, there exists $m_0$ such that
\[
\left|\frac{V(mz)-V(m)}{a(m)}-\frac{z^\gamma-1}{\gamma}\right| \leq \varepsilon \max(z^{\gamma+\delta},z^{\gamma-\delta}) ,\quad m>m_0,  mz>m_0.
\]
We deduce that 
\begin{align*}
&\int_{m_0/m}^\infty \left|\frac{V(mz)-V(m)}{a(m)}-\frac{z^\gamma-1}{\gamma}\right|^p e^{-1/z}\frac{\rmd z}{z^2} \\
&\leq \varepsilon^p \int_0^\infty   \max(z^{(\gamma+\delta)p},z^{(\gamma-\delta)p})e^{-1/z}\frac{\rmd z}{z^2},
\end{align*}
where the last integral is finite for $\delta$ small enough because $\gamma p<1$. On the other hand, the contribution of the interval $[0,m/m_0]$ to the integral defining $A_p''(m)$ is upper bounded as follows. We use the assumption that $F$ has a finite left end point so that $V$ is bounded by $M$ on $[0,1]$, whence
\begin{align*}
&\int_0^{m_0/m} \left|\frac{V(mz)-V(m)}{a(m)}-\frac{z^\gamma-1}{\gamma}\right|^p e^{-1/z}\frac{\rmd z}{z^2} \\
&\leq 2^{p-1}\int_0^{m_0/m}\left( \frac{2M}{a(m)}\right)^pe^{-1/z}\frac{\rmd z}{z^2}+2^{p-1}\int_0^{m_0/m}\left|\frac{z^\gamma-1}{\gamma} \right|^pe^{-1/z}\frac{\rmd z}{z^2} \\
&=  \frac{2^{p-1}(2M)^p}{a(m)^p}e^{-m/m_0} + 2^{p-1}\int_0^{m_0/m}\left|\frac{z^\gamma-1}{\gamma} \right|^pe^{-1/z}\frac{\rmd z}{z^2}.
\end{align*}
Combining these estimates, we get 
\[
A''_p(m)^p= O(\varepsilon^p)+O(a(m)^{-p}e^{-m/m_0})+o(1).  
\]
Since $a$ is regularly varying with index $\gamma$, $a(m)^pe^{-m/m_0}\to 0$ as $m\to\infty$. Letting $m\to\infty$ and $\varepsilon\to 0$, we obtain the convergence  $A_p''(m)\to 0$.\\
\textit{Proof of $iii)$.}
We have
\[
\left(\frac{A_p''(m)}{A(m)}\right)^p=\int_0^\infty \left|\frac{\frac{V(mz)-V(m)}{a(m)}-\frac{z^\gamma-1}{\gamma}}{A(m)} \right|^p e^{-1/z}\frac{\mathrm{d}z}{z^2},
\]
 and the second order condition \eqref{eq:second-order-condition-BM} suggests the limit
 \[
 \left(\frac{A_p''(m)}{A(m)}\right)^p\longrightarrow \int_0^\infty \Psi_{\gamma,\rho}(z)^p e^{-1/z}\frac{\mathrm{d}z}{z^2}=c''_p(\gamma,\rho)^p.
 \]
 The limit is justified as follows. Drees theorem (see \cite{dHF06}, Theorem B.3.10 p.392) implies that for $\varepsilon,\delta>0$, there is $m_0\geq 1$ such that
 \begin{align*}
 &\left| \left(\int_{m_0/m}^\infty \left|\frac{\frac{V(mz)-V(m)}{a(m)}-\frac{z^\gamma-1}{\gamma}}{A(m)} \right|^p e^{-1/z}\frac{\mathrm{d}z}{z^2}\right)^{1/p}
 - \left(\int_{m_0/m}^\infty \Psi_{\gamma,\rho}(z)^p  e^{-1/z}\frac{\mathrm{d}z}{z^2}\right)^{1/p}\right|\\
 &\leq \varepsilon\left(\int_0^\infty \max(z^{p(\gamma+\rho+\delta)},z^{p(\gamma+\rho-\delta)})e^{-1/z}\frac{\mathrm{d}z}{z^2}\right)^{1/p}
 \end{align*}
 where the last integral is finite for $\delta$ small enough because $\gamma p<1$ and $\rho\leq 0$. On the other hand, the contribution of the interval $[0,m_0/m]$ to the integral vanishes as $m\to\infty$ because
\begin{align*}
 &\int_0^{m_0/m} \left|\frac{\frac{V(mz)-V(m)}{a(m)}-\frac{z^\gamma-1}{\gamma}}{A(m)} \right|^p e^{-1/z}\frac{\mathrm{d}z}{z^2}\\
 &\leq  2^{p-1}\left(\frac{2M}{a(m)A(m)}\right)^pe^{-m/m_0}+\frac{2^{p-1}}{A(m)^p}\int_0^{m_0/m}\left|\frac{z^\gamma-1}{\gamma}\right|^p e^{-1/z}\frac{\mathrm{d}z}{z^2}\rightarrow 0.
 \end{align*}
We deduce that $A_p''(m)\sim c_p''(\gamma,\rho)A(m)$ as $m\to\infty$.
\end{proof}

\begin{lemma}\label{lem:W_p-same-weak-limit}
Let $(\mathcal{X},d)$ a polish space. Let $(P_n)_{n\in \mathbb{N}}$ and $(P_n^*)_{n\in \mathbb{N}}$ be two sequences in $\mathcal{W}(\mathcal{X})$ and $P \in \mathcal{W}(\mathcal{X})$. 
If $P_n^*$ converge weakly to $P$ and $W_p(P_n,P_n^*)\rightarrow 0$ then $P_n$ converge weakly to $P$.
\end{lemma}

\begin{proof}
	Let $f$ be a real valued bounded Lipschitz function. Let $\varepsilon >0$ and consider $(X_n,X_n^*)$ a coupling between $P_n$ and $P_n^*$ such that 
	\[
	\|d(X_n,X_n^*)\|_{L^p} \leq (1+\varepsilon)W_p(P_n,P_n^*).
	\]
	Since $W_p(P_n,P_n^*)\rightarrow 0$  as $n\to\infty$,  $\|d(X_n,X_n^*)\|_{L^p} \rightarrow 0$ and also $\mathbb{P}[d(X_n,X_n^*)]\rightarrow 0$.  As a consequence, writing $\mathrm{Lip}(f)$ for the Lipschitz constant of $f$,
	\begin{align}\label{eq:proof_lem1}
	\mathbb{E}[|f(X_n)-f(X_n^*)|]&\leq \mathrm{Lip}(f) \mathbb{E}[d(X_n,X_n^*)]. 
	\end{align}
Since $P_n^*\stackrel{d}\rightarrow P$, we have $\mathbb{E}f(X_n^*)\rightarrow  \mathbb{E}f(X)$ where $X\rightsquigarrow P$ and Equation~\eqref{eq:proof_lem1} implies  $\mathbb{E}f(X_n)\rightarrow  \mathbb{E}f(X)$. Since $f$ is an arbitrary bounded Lipschitz function, we conclude by the Portmanteau Theorem that $X_n$ converges weakly to $X$ as $n\to\infty$, that is  $P_n\stackrel{d}\rightarrow P$.
\end{proof}

\bibliographystyle{apalike}
\bibliography{biblio}

\end{document}